\documentclass[11pt,a4paper]{amsart}
\usepackage{amssymb,amsfonts,amsmath,mathrsfs}
\usepackage[left=2cm,right=2cm,top=2cm,height=250mm]{geometry}
\usepackage{color}
\usepackage{verbatim}
\newtheorem{proposition}{Proposition}[section]
\newtheorem{theorem}[proposition]{Theorem}
\newtheorem{corollary}[proposition]{Corollary}
\newtheorem{lemma}[proposition]{Lemma}

\theoremstyle{definition}
\newtheorem{definition}[proposition]{Definition}

\newtheorem{remark}[proposition]{Remark}

\numberwithin{equation}{section}

\def\R{\Bbb R}
\def\Dx{\Delta_x}
\def\Nx{\nabla_x}
\def\Dt{\partial_t}
\def\({\left(}
\def\){\right)}
\def\eb{\varepsilon}
\def\Cal{\mathcal}
\def\Bbb{\mathbb}
\def\divv{\operatorname{div}}
\begin{document}
\title[Navier-Stokes equations in a strip]{Infinite-energy solutions for the Navier-Stokes equations in a strip revisited}
\author[] {Peter Anthony${}^1$, and Sergey Zelik${}^1$}

\begin{abstract} The paper deals with the Navier-Stokes equations in a strip in the class of spatially non-decaing (infinite-energy) solutions belonging to the properly chosen uniformly local Sobolev spaces. The global well-posedness and dissipativity  of the Navier-Stokes equations in a strip in such spaces has been first established in \cite{ZelikGlasgow}. However, the proof given there contains rather essential error and the aim of the present paper is to correct this error and to show that the main results of \cite{ZelikGlasgow} remain true.
\end{abstract}

\thanks{
This work is partially supported by the Russian Ministry of Education and Science (contract no.
8502).}

\address{${}^1$ University of Surrey, Department of Mathematics, \newline
Guildford, GU2 7XH, United Kingdom.}

\email{p.antony@surrey.ac.uk}
\email{s.zelik@surrey.ac.uk}
\maketitle
\tableofcontents
\section{Introduction}\label{s0}
We study the infinite energy solutions of the Navier-Stokes equations
\begin{equation}\label{0.eqmain}
\begin{cases}
\Dt u+(u,\Nx)u+\Nx p=\Dx u+g,\\
\divv u=0,\ \ u\big|_{\partial\Omega}=0,\ \ u\big|_{t=0}=u_0
\end{cases}
\end{equation}
in a strip $\Omega=\R\times(-1,1)$. Note that the case where the solution $u=(u_1,u_2)$ has the finite energy is well understood now-a-days, see \cite{Ab1,babin,babin1,temam,temam1} and also references therein. In that case the basic energy estimate can be obtained by multiplication of \eqref{0.eqmain} by $u$, integrating over $\Omega$ and using the fact that
\begin{equation}\label{0.0}
\int_{\Omega}(u(x),\Nx)u(x).u(x)\,dx=0
\end{equation}
for any (square integrable) divergence free function $u$ satisfying the Dirichlet boundary conditions. However, the most interesting from the physical point of view solutions of problem \eqref{0.eqmain} naturally have {\it infinite} energy, for instance, it will be so for the classical Poiseille flow
$$
u(x)=\(\begin{matrix} \alpha(x_2^2-1)\\0\end{matrix}\),\  \alpha\in\R
$$
as well as for all other solutions bifurcating from it, therefore, exactly the infinite energy solutions look as a relevant class of solutions here from the physical point of view.
\par
The theory of dissipative dynamical systems in unbounded domains and associated infinite energy solutions are intensively developing during the last 20 years starting from the pioneering papers \cite{BabinVishik,Ab1,Ab2}, see also \cite{MielkeS,EfZelCPAM,ZelikCPAM,MirZel} and references therein. In this theory, the so-called {\it uniformly-local} Sobolev spaces defined via
$$
W^{l,p}_b(\Omega):=\{u\in \Cal D'(\Omega),\ \|u\|_{W^{l,p}_b}:=\sup_{s\in\R}\|u\|_{W^{l,p}(\Omega_s)}<\infty\},\ \Omega_s:=(s,s+1)\times(-1,1)
$$
are used as the phase spaces for the problems considered. Indeed, on the one hand, in contrast to the usual Sobolev spaces, these spaces contain constants, space-periodic solutions, etc. and look more suitable for the case of unbounded domains. On the other hand, in these spaces one has the regularity theory for the elliptic/parabolic equations which is very similar to the one developed for the usual Sobolev spaces, see e.g., \cite{MirZel} and also Section \ref{s1} below. Note also that, in order to obtain the proper estimates for the solutions in the uniformly local spaces, one can use the so-called {\it weighted} energy estimates as an intermediate step and utilize the relation
$$
\|u\|_{L^2_b}\sim \sup_{s\in\R}\|u\|_{L^2_{\phi(\cdot-s)}},
$$
where $\phi$ is a properly chosen (square integrable) weight function, see Section \ref{s1} for more details.
\par
According to this strategy, in order to obtain the estimates for the solutions of \eqref{0.eqmain} in the uniformly local spaces (say, in $[L^2_b(\Omega)]^2$), it would be natural to try to multiply equation \eqref{0.eqmain} by $\phi^2u$, where $\phi=\phi(x_1)$ is a properly chosen weight function. However, there are two principal difficulties arising here. First, we do not have the analogue of \eqref{0.0} for the weighted case
$$
\int_\Omega (u(x),\Nx)u(x).\phi^2(x_1)u(x)\,dx=-2\int_{\Omega}\phi(x_1)\phi'(x_1) u_1(x)|u(x)|^2\,dx\ne0,
$$
so the non-linear term does not vanish, but produces the extra cubic term which should be somehow estimated (this is far from being straightforward since the  "good" terms in the weighted energy inequality are only quadratic). Second, the function $\phi^2u$ is no more divergent free
$$
\divv u=2\phi\phi'u_1\ne0,
$$
so the term containing pressure also survives in the weighted energy estimate and requires to be controlled. The situation is simpler in the case $\Omega=\R^2$ or $\Omega=\R\times(-1,1)$ with the periodic boundary conditions when the maximum principle can be applied to the vorticity equation which gives an important extra estimate, see \cite{Giga1,Giga2,MielkeA,ZelJMFM} and the references therein for more details (see also \cite{L02} for some estimates in uniformly local spaces in the 3D case $\Omega=\R^3$).
\par
An effective way to overcome both of the aforementioned problems has been suggested in \cite{ZelikGlasgow} (see also \cite{ZelikInst}) where the the weighted energy theory for the Navier-Stokes equations in cylindrical domains has been developed. The problem with the extra cubic term has been solved there by using the special weights
$$
\theta_{\eb,s}(x):=\frac1{\sqrt{1+\eb^2|x-s|^2}}, \ \ s\in\R,
$$
depending on a small parameter $\eb>0$. Then, since these weights satisfy
$$
|\theta_{\eb,s}'(x)|\le C\eb[\theta_{\eb,s}(x)]^2,
$$
the nonlinearity produces only the small extra term of the form $\eb\|u\|^3_{L^3_\theta}$ which can be then controlled by the proper choice of the parameter $\eb$ depending on the initial condition $u_0$, see also Section \ref{s3}.
\par
To solve the second problem, it was suggested to multiply equation \eqref{0.eqmain} by $\theta^2u-v_\theta$, where the corrector $v_\theta$ solves the following linear adjoint problem:
\begin{equation}\label{0.corr}
-\Dt v_\theta+\Nx q=\Dx v_\theta,\ \ v_\theta\big|_{\partial\Omega}=0,\ \ \divv v_\theta=2\theta_{\eb,s}\theta'_{\eb,s} u_1.
\end{equation}
Then, since $\divv(\theta^2u-v_\theta)=0$, the pressure term vanishes and one has the following weighted energy equality:
\begin{equation}\label{0.energy}
\frac d{dt}\(\frac12\|u(t)\|^2_{L^2_\theta}-(u(t),v_\psi(t))\)+(\Nx u(t),\Nx(\theta^2u(t))=(g-(u(t),\Nx)u(t),\theta^2u(t)-v_\theta(t)).
\end{equation}
Here and below $(u,v)$ stands for the standard inner product in $L^2(\Omega)$.
Moreover, since the data for $v_\theta$ contains the multiplier $\theta'\sim\eb\theta^2$, the corrector $v_\theta$ should be small (at least, of order~$\eb$) and, by this reason should not destroy the energy estimate. The realization of this strategy in \cite{ZelikGlasgow} gave the following result.
\begin{theorem}\label{Th0.main} For any external force $g\in[L^2_b(\Omega)]^2$ and any divergent free  $u_0\in [L^2_b(\Omega)]^2$, $\partial_{n}u_0\big|_{\partial\Omega}=0$, there exists a unique solution $u(t)\in [L^2_b(\Omega)]^2$ of the Navier-Stokes problem \eqref{0.eqmain} satisfying the mean flux condition:
$$
\int_{-1}^1u_1(t,x_1,x_2)\,dx_2=\int_{-1}^1u_1(0,x_1,x_2)\,dx_2=c\in\R
$$
and the following dissipative estimate holds:
\begin{equation}\label{0.est}
\|u(t)\|_{L^2_b}\le Q(\|u_0\|_{L^2_b})e^{-\alpha t}+C(1+c^3+\|g\|_{L^2_b}^2),
\end{equation}
where the positive constants $\alpha$ and $C$ and the monotone function $Q$ are independent of $u_0$, $g$ and $t$.
\end{theorem}
Clearly, the  estimates for the auxiliary problem \eqref{0.corr} play an important role in the proof of this theorem. Namely, as stated in Theorem 5.1 of \cite{ZelikGlasgow}, the corrector $v_\theta$ satisfies the following estimate:
\begin{equation}\label{0.wrong}
\|v_\theta\|_{C(0,T;W^{1,2}_{\theta^{-2}})}+\|v_\theta\|_{L^2(0,T;W^{2,2}_{\theta^{-2}})}\le C\eb(\|u\|_{C(0,T;L^2_\theta)}+\|u\|_{L^2(0,T;W^{1,2}_\theta)}).
\end{equation}
Unfortunately, the proof of this key estimate contains an essential error. Namely, the functon $\bar w+\Pi v$ involved into equation (5.14) at page 553 of \cite{ZelikGlasgow} has {\it non-zero} boundary conditions (although $v\big|_{\partial\Omega}=0$, $\Pi v\big|_{\partial\Omega}\ne0$ since the Leray projector does not preserve Dirichlet boundary conditions) and, by this reason, the multiplication of this equation by $\divv(\varphi_{2\mu,x_0}\Nx(\bar w+\Pi v))$ leads to the extra uncontrollable boundary term which is missed in equality (5.15). Thus, the proof of Theorem 5.1 given in \cite{ZelikGlasgow} is formally wrong. Moreover, estimate \eqref{0.wrong} is probably  wrong as well (at least, we do not know how to prove good estimates for the auxiliary problem \eqref{0.corr} in the weighted  space $C(0,T;W^{1,2}_{\phi}(\Omega))$ which are stated in this theorem and only the weaker versions of these estimates, e.g., in the space $C(0,T; L^3_{\phi}(\Omega))$ are available, see Section \ref{s2} below) .
\par
The aim of the present paper is to correct the aforementioned error and to show that  Theorem \ref{Th0.main} stated above remains true. To this end,
we first develop an alternative approach to study the auxiliary equation \eqref{0.corr} based more on the methods of the analytic semigroup theory rather than on energy estimates and verify a weaker version of Theorem 5.1 from \cite{ZelikGlasgow}. In a fact, we are unable to establish the control of the corrector $v_\theta$ in $C(0,T;W^{1,2}_{\theta^{-2}})$ or in $L^2(0,T;W^{2,2}_{\theta^{-2}})$ (as stated in this theorem), but the following weaker version of these estimates hold:
\begin{equation}\label{0.weak}
\|v_\theta\|_{C(0,T;L^2_{\theta^{-2}})}+\|v_\theta\|_{C(0,T;L^3_{\theta^{-2}})}\le C\eb\|u\|_{C(0,T;L^2_\theta)},
\end{equation}
see Section \ref{s2} for the details. Then, keeping in mind that estimate \eqref{0.weak} is essentially weaker than the original estimate \eqref{0.wrong} used in \cite{ZelikGlasgow}, we have to rework most part of proofs  given in \cite{ZelikGlasgow} in order to show that this loss of the regularity of the corrector $v_\theta$ is not crucial and that the main results remain true despite the aforementioned error.
\par
The paper is organized as follows. The definitions of the proper weight functions and associated weighted spaces as well as their basic properties are given in Section \ref{s1}. Moreover, we briefly recall here the known facts on the regularity of the Leray projector and the Stokes operator in these spaces. The auxiliary linear problem  \eqref{0.corr} is studied in Section \ref{s2}. In particular, estimate \eqref{0.weak} as well as the energy identity \eqref{0.energy} are verified there. The proof of the key estimate \eqref{0.est} (in the non-dissipative form with $\alpha=0$) is verified in Section \ref{s3}. The existence and uniqueness of a weak solution for the Navier-Stokes problem \eqref{0.eqmain} in the uniformly local spaces is verified in Section \ref{s4} and, finally, the dissipative version of estimate \eqref{0.est} (with $\alpha>0$) and the parabolic smoothing property for the weak solutions of \eqref{0.eqmain} are proved in Section \ref{s5}.

\section{Preliminaries}\label{s1}
In that section, we briefly recall the definitions and key properties of weights and weighted Sobolev spaces and state a number of known results on the
regularity of the Leray projector and Stokes operator in these spaces which are crucial for what follows, see \cite{babin1, ZelikGlasgow} for the detailed exposition. We start by defining the  class of admissible weights and associated weighted spaces adopted to the case of the strip $\Omega=\R\times(-1,1)$.

\begin{definition} \label{Def1.weights} A function $\phi(x)$, $x\in\R$, is a weight function of exponential growth rate $\mu>0$ if
\begin{equation}\label{1.wexp}
\phi(x)>0,\ \ \ \phi(x+y)\le Ce^{\mu|y|}\phi(x)
\end{equation}
holds for all $x,y\in\R$. The weighted Lebesgue space $L^p_\phi(\Omega)$, $1\le p\le\infty$, is defined as a subspace of $L^p_{loc}(\Omega)$ for which the following norm is finite:
\begin{equation}\label{1.wlp}
\|u\|_{L^p_\phi}:=\(\int_\Omega \phi^p(x_1)|u(x)|^p\,dx\)^{1/p},
\end{equation}
where $x=(x_1,x_2)\in\Omega$. The uniformly local Lebesgue space $L^p_b(\Omega)$ is determined by the finiteness of the following norm:
\begin{equation}
\|u\|_{L^p_b}:=\sup_{s\in\R}\|u\|_{L^p(\Omega_s)},
\end{equation}
where $\Omega_s:=(s,s+1)\times(0,1)$. As usual, the weighted ($W^{l,p}_\phi(\Omega)$) and uniformly local ($W^{l,p}_b(\Omega)$) Sobolev spaces are defined as spaces of distributions whose derivatives up to order $l$ belong to $L^p_\phi(\Omega)$ (resp. $L^p_b(\Omega)$). This definition works for $l\in\Bbb N$ and, for the non-integer or negative $l$th, the corresponding Sobolev spaces can be defined via the interpolation and duality arguments, see \cite{EfZelCPAM,ZelikCPAM,MirZel} for the details.
\par
We will also need the uniformly local spaces for the functions $u(t,x)$, $x\in\Omega$, depending also on time $t\in\R$, so, for every $1\le p\le\infty$, we define the space $L^p_b(\R\times\Omega)$ by the following norm:
\begin{equation}
\|u\|_{L^p_b(\R\times\Omega)}:=\sup_{(t,s)\in\R^2}\|u\|_{L^p((t,t+1)\times\Omega_s)}
\end{equation}
and the spaces $L^p_b((A,B)\times\Omega)$ are defined analogously. More general, for $1\le p,q\le\infty$, we defined the space $L^q_b(\R,L^p_b(\Omega)$ by the following norm:
$$
\|u\|_{L^q_b(\R,L^p_b(\Omega))}:=\sup_{(t,s)\in\R^2}\|u\|_{L^q(t,t+1;L^p(\Omega_s))}.
$$
\end{definition}
The natural choices of the weights of exponential growth rate are the following ones:
$$
\varphi_{\eb,x_0}(x):=e^{-\eb|x-x_0|},\ \ \bar\varphi_{\eb,x_0}(x):=e^{\sqrt{\eb^2|x-x_0|^2+1}}
$$
which, obviously, have the exponential growth rate $|\eb|$ or the polynomial weights, e.g.,
\begin{equation}\label{1.thetaweight}
\theta_{\eb,x_0}(x):=\frac1{\sqrt{1+\eb^2|x-x_0|^2}}
\end{equation}
This weight, in addition, to \eqref{1.wexp} (which holds for every positive $\mu$), satisfies the following property:
\begin{equation}\label{1.thetaeb}
|\theta_{\eb,x_0}'(x)|\le C\eb\theta_{\eb,x_0}(x)^2\le C\eb\theta_{\eb,x_0}(x)
\end{equation}
which is crucial for what follows. A bit more general are the weights $\theta_{\eb,x_0}(x)^N$, $N\in\R$, $N\ne0$, which are also the weights of exponential growth rate $\mu$ for any $\mu>0$ and satisfy the analog of \eqref{1.thetaeb} where the exponent $2$ is replaced by $\frac{N+1}N$.
\par
The next proposition which gives the equivalent representation of the weighted Sobolev norms in terms of the non-weighted ones is very useful in many estimates.

\begin{proposition}\label{Prop1.wequiv} Let $\phi$ be the weight of exponential growth rate, $1\le p<\infty$ and $l\in\R$. Then
\begin{equation}\label{1.wequiv}
C_2\|u\|^p_{W^{l,p}_\phi(\Omega)}\le\int_{s\in\R}\phi^p(s)\|u\|^p_{W^{l,p}(\Omega_s)}\,ds\le C_1\|u\|^p_{W^{l,p}_\phi(\Omega)},
\end{equation}
where the constants $C_i$ depend only on $\mu$ and constant $C$ involved in \eqref{1.wexp} and are independent of the concrete choice of the weight $\phi$.
\end{proposition}
For the proof of this proposition, see e.g., \cite{EfZelCPAM} or \cite{ZelikCPAM}.
\par
The next proposition which connects the weighted and uniformly local norms is the main technical tool for obtaining the estimates of solutions in uniformly local spaces.
\begin{proposition}\label{Prop1.wb} Let $\phi$ be the weight function of exponential growth rate such that $\phi\in L^p(\R)$ and let $1\le p<\infty$. Then,
\begin{equation}\label{1.wb}
C_1\|\phi\|_{L^p}^{-1}\|u\|_{W^{l,p}_\phi(\Omega)}\le\|u\|_{W^{l,p}_b}\le C_2\sup_{s\in\R}\|u\|_{W^{l,p}_{\phi(\cdot-s)}(\Omega)},
\end{equation}
where the constants $C_1$ and $C_2$ depend only on $p$ and $C$ and $\mu$ involved in \eqref{1.wexp} and are independent of the concrete choice of the weight $\phi$.
\end{proposition}
For the proof of this proposition, see \cite{ZelikCPAM}.
\par
We will essentially use the particular case of this estimate with $p=2$ and $\phi=\theta_{\eb,x_0}$ where $\eb\ll1$. Then the left-hand side of \eqref{1.wb} gives
\begin{equation}\label{1.wbeb}
\|u\|_{L^2_{\theta_{\eb,x_0}}}\le C\eb^{-1/2}\|u\|_{L^2_b},
\end{equation}
where $C$ is independent of $\eb$ and $x_0$.
\par
At the next step, we introduce the standard (for the theory of the Navier-Stokes equations) spaces $\Cal H$, $\Cal V$ and $\Cal V^*$:
$$
\Cal H:=[u\in [C^\infty_0(\Omega)]^2,\ \divv u=0]_{[L^2(\Omega)]^2},\ \ \Cal V:=[u\in [C^\infty_0(\Omega)]^2,\ \divv u=0]_{[W^{1,2}(\Omega)]^2},
$$
where $[\cdot]_V$ means the closure in the space $V$ and $\Cal V^*$ stands for the dual space to $\Cal V$ (with respect to the inner product in $\Cal H$). The spaces $\Cal H_\phi$ and $\Cal V_\phi$ are defined analogously (the closure is taken in $[L^2_\phi(\Omega)]^2$ and $[W^{1,2}_\phi(\Omega)]^2$ respectively)

 The following proposition describes the structure of the introduced spaces.

\begin{proposition}\label{Prop1.stream} The spaces $\Cal H$ and $\Cal V$ can be described as follows:
\begin{equation}\label{1.HV}
\Cal H=\{u\in[L^2(\Omega)]^2,\ \Bbb Su_1\equiv0,\ \ \divv u=0,\ \ u_2\big|_{\partial\Omega}=0\},\ \ \Cal V=\Cal H\cap[H^1_0(\Omega)]^2,
\end{equation}
where $\Bbb Sv(x_1):=\frac12\int_{-1}^1v(x_1,x_2)\,dx_2$ and $H^1_0(\Omega):=W^{1,2}(\Omega)\cap\{u\big|_{\partial\Omega}=0\}$. Moreover, let for any
$u\in\Cal H$ (resp. $u\in\Cal V$)
\begin{equation}
\psi:=\Psi(u)=\int_{-1}^{x_2}u_1(x_1,s)\,ds
\end{equation}
be the associated stream function. Then,
$$
u_1=\partial_{x_2}\psi,\ \ u_2=-\partial_{x_1}\psi
$$
and the operator $\Psi$ realizes the isomorphism between $\Cal H$ and $H^1_0(\Omega)$ (resp. between $\Cal V$ and $H^2_0(\Omega)$). Furthermore, for any weight $\phi$ of exponential growth rate, the analogous desription and the analogous isomorphism works for the weighted spaces $\Cal H_\phi$ and $\Cal V_\phi$ as well.
\end{proposition}
For the proof of this proposition, see e.g., \cite{ZelikGlasgow}.
\par
Ananlogously to \eqref{1.HV}, we define the {\it uniformly local} spaces $\Cal H_b$ and $\Cal V_b$ via
\begin{equation}\label{1.HVB}
\Cal H_b=\{u\in[L^2(\Omega)]^2,\ \Bbb Su_1\equiv0,\ \ \divv u=0,\ \ u_2\big|_{\partial\Omega}=0\},\ \ \Cal V_b=\Cal H_b\cap[W^{1,2}_b(\Omega)]^2\cap\{u\big|_{\partial\Omega}=0\},
\end{equation}
Note that, in contrast to the case of spaces $\Cal H$ or $\Cal H_\phi$, the uniformly local spaces $\Cal H_b$ and $\Cal V_b$ do not coincide with the closure of $C^\infty_0(\Omega)$ in the proper uniformly local norms. However, the operator $\Psi$ is still the isomorphism between the corresponding uniformly local spaces, see \cite{ZelikGlasgow} for more details.
\par
We now recall that the space $\Cal H$ is orthogonal to any gradient vector field and, due to the Leray-Helmholtz decomposition, any vector field
$u\in[L^2(\Omega)]^2$ can be presented in a unique way as a sum
\begin{equation}
u=v+\Nx p,\ \ v\in\Cal H.
\end{equation}
Therefore, the Leray (ortho)projector $\Pi:[L^2(\Omega)]^2\to\Cal H$ is well defined via $\Pi u:=v$. We also recall that the Stokes operator is defined as the following  self-adjoint operator in $\Cal H$:
\begin{equation}\label{1.stokes}
A:=-\Pi\Dx,\ \ D(A)=\Cal V\cap [H^2(\Omega)]^2.
\end{equation}
The next standard proposition gives the description of domains of its fractional powers. This result will be used in the next section for deriving the weighted energy estimates.
\begin{proposition}\label{Prop1.fracS} Let $\kappa\in(0,1)$. Then the domain of $A^\kappa$ possesses the following description:
\begin{equation}\label{1.fracS}
D(A^\kappa)=[D((-\Dx)^\kappa)]^2\cap \Cal H,
\end{equation}
where $D((-\Dx)^\kappa)$ is the domain of the fractional Laplacian with Dirichlet boundary conditions.
\end{proposition}
For the proof of this result, see e.g., \cite{Amann}, see also \cite{Abels, Judovich} where the analogous result is obtained not only for $L^2$, but also for the $L^p$-spaces, $1<p<\infty$.
\par
Since, the description of the domains of the fractional Laplacian is well-known, see e.g., \cite{triebel},
\begin{equation}\label{1.fracL}
D((-\Dx)^\kappa)=
\begin{cases}
W^{2\kappa,2}(\Omega),\ \ \kappa<1/4;\\
W^{2\kappa,2}(\Omega)\cap\{u\big|_{\partial\Omega}=0\},\ \ \kappa>1/4;\\
W^{2\kappa,2}(\Omega)\cap\{\int_{\Omega}\frac1{1-x_2^2}|u(x)|^2\,dx<\infty\},\ \ \kappa=1/4,
\end{cases}
\end{equation}
Proposition \ref{Prop1.fracS} gives the description of $D(A^\kappa)$ in terms of the usual Sobolev spaces.
\par
The next result gives the regularity of the Leray projector and the Stokes operator in weighted and uniformly local Sobolev spaces.

\begin{proposition}\label{Prop1.wLStokes}Let $\phi$ be the weight of a sufficiently small exponential growth rate. Then, for any $l\ge0$ and $1<p<\infty$, the Leray projector $\Pi$ can be extended in a unique way by continuity to the continuous operator
\begin{equation}\label{1.wp}
\Pi:[W^{l,p}_\phi(\Omega)]^2\to [W^{l,p}_\phi(\Omega)]^2
\end{equation}
and the norm of this operator depends  on $l$, $p$ and the constant $C$ involved into the inequality \eqref{1.wexp} and is uniformly bounded with respect to the concrete choice of $\phi$. Furthermore, analogously, the Stokes operator $A$ can be extended to the isomorphism
\begin{equation}\label{1.ws}
A: [W^{l+2,2}_\phi(\Omega)]^2\cap\Cal H_\phi\cap \{u\big|_{\partial\Omega}=0\}\to [W^{l,2}_\phi(\Omega)]^2\cap\Cal H_\phi
\end{equation}
and the norms of $A$ and $A^{-1}$ are uniformly bounded with respect to the concrete choice of $\phi$. Moreover, the analogous results hold for the uniformly local Sobolev spaces as well.
\end{proposition}
The proof of this result can be found, e.g., in \cite{ZelikGlasgow}, see also \cite{babin}.
\par
We also state the analogue of the regularity result for the Stokes operator in negative Sobolev spaces.
\begin{proposition}\label{Prop1.stokes} Let $\phi$ be the weight function of sufficiently small exponential growth rate. Then, for every $g\in [W^{-1,2}(\Omega)]^2$, there is a unique solution $u\in\Cal V_\phi$ of the Stokes problem
\begin{equation}
\Dx u-\Nx p=g,\ \ \divv u=0
\end{equation}
and the following estimate holds:
\begin{equation}\label{1.stokesm1}
\|u\|_{\Cal V_\phi}\le C\|g\|_{[W^{-1,2}_\phi(\Omega)]^2},
\end{equation}
where the constant $C$ is independent of the concrete choice of the weight $\phi$. The analogous result holds also for the uniformly local spaces.
\end{proposition}
The proof of this result  can be found in \cite{ZelikGlasgow} and \cite{babin}.
\par
We conclude this section by given the result on solvability of the non-stationary Stokes problem in weighted spaces for the case of strong solutions.
\begin{proposition}\label{Prop1.strongw} Let $\phi$ be the weight function of sufficiently small exponential growth rate. Then, for every $g\in L^2(0,T;L^2_\phi(\Omega))$ and every $u_0\in\Cal V_\phi$, there is a unique solution $u(t)$ of the problem
\begin{equation}\label{1.NStokes}
\Dt u-\Dx u+\Nx p=g(t),\ \divv u=0,\ \ u\big|_{\partial\Omega}=0,\ u\big|_{t=0}=u_0
\end{equation}
which satisfies $\Dt u,\Dx u\in L^2(0,T;L^2_\phi)$ and the following estimate holds:
\begin{multline}\label{1.nonS}
\|u(t)\|_{L^2_\phi}^2+\|\Dt u\|^2_{L^2(\max\{0,t-1\},t;L^2_\phi)}+\|\Dx u\|^2_{L^2(\max\{0,t-1\},t;L^2_\phi)}\le \\\le C\|u_0\|^2_{L^2_\phi}e^{-\alpha t}+\int_0^te^{-\alpha(t-s)}\|g(s)\|^2_{L^2_\phi}\,ds,
\end{multline}
where $C$ and $\alpha>0$ are independent of $t\ge0$, $u_0$, $g$ and the concrete choice of the weight $\phi$.
\end{proposition}
The proof of this result also can be found in \cite{ZelikGlasgow} or \cite{babin}.
\begin{remark}\label{Rem1.dif} Most results of this section will be used in the sequel with the weights $\phi=\theta_{\eb,x_0}(x_1)$ only. However, we will need to control the dependence of all the constants on the parameter $\eb\to0$ and, by this reason, it is important for us that the constants in the above propositions are "independent of the concrete choice of the weight" and depend only on the constants in \eqref{1.wexp} (which are uniform with respect to $\eb\to0$ and $x_0\in\R$).
\par
Another straightforward observation which will be essentially used in the next section is that, according to Propositions \ref{Prop1.fracS} and \ref{Prop1.wLStokes} and formula \eqref{1.fracL}, the Leray projector $\Pi$ maps $[W^{2\kappa,2}(\Omega)]^2$ to $D(A^\kappa)$:
\begin{equation}\label{1.kappaPi}
\Pi:[W^{2\kappa,2}(\Omega)]^2\to D(A^\kappa)
\end{equation}
for $\kappa<1/4$ and that is not true for $\kappa\ge1/4$ due to the loss of zero boundary conditions (we recall that, in general, $\Pi u\big|_{\partial\Omega}\ne0$ even if $u\big|_{\partial\Omega}=0$).
\par
Finally, it worth to mention that the dual space $\Cal V^*$ is {\it not} a subspace of distributions $[D'(\Omega)]^2$ and the fact that  the divergence free vector field $v\in\Cal V^*$ {\it does not} imply in general that its components $v_1$ or $v_2$ belong to $H^{-1}(\Omega)$. Indeed, we may add any gradient of a harmonic function to the vector field $v$ without changing the functional $v\in\Cal V^*$ and this harmonic function may be not in $L^2(\Omega)$, say, due to the singularities near the boundary. As will be explained in the next section, this leads to essential difficulties in developing the weighted energy theory for the non-stationary Stokes equations.
\end{remark}

\section{The linear non-stationary Stokes equation: weighted energy theory}\label{s2}
The aim of this section is derive the so-called weighted energy equality for the  linear non-autonomous Stokes problem \eqref{1.NStokes} under the assumptions that
\begin{equation}\label{2.gas}
g\in L^{4/3}_b(\R_+\times\Omega)\cap L^1_b(\R_+,L^{3/2}_b(\Omega)),\ \ u_0\in\Cal H_b
\end{equation}
which will be used in the sequel for the study of the nonlinear Navier-Stokes equation.

\begin{definition}\label{Def2.en} A function $u(t,x)$ is a weak  (energy) solution of \eqref{1.NStokes} if
\begin{equation}\label{2.defspace}
u\in L^\infty(0,T;\Cal H_b)\cap C(0,T;\Cal H_\phi),\ \ \Nx u\in L^2_b([0,T]\times\Omega),
\end{equation}
where $\phi$ is any weight of exponential growth rate such that $\phi\in L^2(\R)$ and $u$ solves \eqref{1.NStokes} in the sense of distributions, namely, for any $\varphi\in C^\infty_0([0,T]\times\Omega)$ satisfying $\divv \varphi=0$,
$$
-\int_{\R_+}(u,\Dt\varphi)\,dt -\int_{\R_+}(u,\Dx\varphi)\,dt=\int_{\R_+}(g,\varphi)\,dt.
$$
Here and below $(u,v)$ stands for the standard inner product in $[L^2(\Omega)]^2$.
\end{definition}
To derive the weighted energy equality for problem \eqref{1.NStokes}, it would be natural to multiply the equation by $\phi^2 u$ for some properly chosen weight $\phi$ of exponential growth rate. However, this does not work in a straightforward way since
$$
\divv (\phi^2 u)=2\phi\phi' u_1\ne0.
$$
To overcome this difficulty, we introduce (following \cite{ZelikGlasgow}) the corrector $v_\phi$ as a solution of the following auxiliary problem:
\begin{equation}\label{2.au}
-\Dt v_\phi-\Dx v_\phi+\Nx q=0,\ \ \divv v_\phi=2\phi\phi' u_1,\ \ v_\phi\big|_{\partial\Omega}=0,\ \ v_\phi\big|_{t=T}=(0,2\phi\phi'\Psi(u(T))),
\end{equation}
where $T>0$ is a parameter and $\psi=\Psi(u)$ is a stream function of the vector field $u$. Then, $\divv(\phi^2v-v_\phi)=0$ and we may at least formally multiply equation \eqref{1.NStokes} by $\phi^2u-v_\phi$ without taking a special care on the pressure term $\Nx p$. Note also that
since \eqref{2.au} is the adjoint equation to \eqref{1.NStokes}, we need to solve it backward in time (for $t\le T$) and
 the unusual initial data at $t=T$ is chosen in order to satisfy the necessary compatibility condition
$$
\divv u\big|_{t=T}=\divv u(T)=2\phi\phi'u_1(T).
$$
Then, multiplying formally equation \eqref{1.NStokes} by $\phi^2u-v_\phi$ and integrating over $x$, we have
\begin{multline}\label{2.formal}
(\Dt u-\Dx u-\Nx p,\phi^2u-v_\phi)=\frac d{dt}(\frac12\|u\|^2_{L^2_\phi}-(u,v_\phi))+(u,-\Dt v_\phi-\Dx v_\phi)+\\+(\Nx u,\Nx(\phi^2u))=\frac d{dt}(\frac12\|u\|^2_{L^2_\phi}-(u,v_\phi))-(u,\Nx q)+(\Nx u,\Nx(\phi^2 u))=\\=\frac d{dt}(\frac12\|u\|^2_{L^2_\phi}-(u,v_\phi))+(\Nx u,\Nx(\phi^2 u)),
\end{multline}
where we have used that $\divv u=\divv(\phi^2u-v_\phi)=0$. Thus, we formally end up with the key weighted energy identity
\begin{equation}\label{2.energy}
\frac d{dt}\(\frac12\|u\|^2_{L^2_\phi}-(u,v_\phi)\)+(\Nx u,\Nx(\phi^2u))=(g,\phi^2u-v_\phi).
\end{equation}
The main aim of this section is to justify the weighted energy equality \eqref{2.energy}. To this end, we first need to study the solutions of the auxiliary problem \eqref{2.au}. For simplicity, we switch back to forward in time solutions by the change of time variable $t\to T-t$ and consider slightly more general problem
\begin{equation}\label{2.aug}
\Dt v-\Dx v+\Nx q=0,\ \ \divv v=\partial_{x_2}h(t),\ \ v\big|_{\partial\Omega}=0,\ \ v\big|_{t=0}=(0,h(0))^t,
\end{equation}
where the function $h\in C(0,T;W^{1,2}_\phi(\Omega))$ for some weight $\phi$ of exponential growth rate which satisfies the additional assumption
\begin{equation}\label{2.w}
|\phi'(x)|+|\phi''(x)|+|\phi'''(x)|\le C\eb\phi(x),
\end{equation}
and the parameter $\eb>0$ is small enough. Then, the choice
\begin{equation}\label{2.relation}
h(t)=2\phi\phi' \Psi(u_1(t))
\end{equation}
corresponds to the considered auxiliary problem \eqref{2.au}. The next theorem which gives the estimate for the solution $v$ in terms of the function $h$ is crucial for what follows.
\begin{theorem}\label{Th2.au} Let $\phi$ be a weight of exponential growth rate which satisfies \eqref{2.au} for sufficiently small $\eb\le\eb_0$ and let $h\in C(0,T;W^{1,2}_\phi(\Omega))$. Then, there exists a unique solution $v(t)$ of problem \eqref{2.au} such that
\begin{equation}\label{2.strange}
v\in C(0,T;W^{1/3,2}(\Omega)),\ \int_0^tv(t)\,dt\in L^2(0,T;W^{2+1/3,2}_\phi(\Omega)),\ \ \Bbb Sv_1\equiv0
\end{equation}
and the following estimate holds:
\begin{equation}\label{2.mainest}
\|v\|_{C(0,T;W^{1/3,2}_\phi(\Omega))}\le C\|h\|_{C(0,T;W^{1,2}_\phi(\Omega))},
\end{equation}
where the constant $C$ is independent of $h$, $T$, $\eb$ and the choice of the weight $\phi$.
\end{theorem}
\begin{remark} \label{Rem2.strange} The regularity assumptions \eqref{2.strange} are a bit unusual. Indeed, it would be more natural to expect
similar to Proposition \ref{Prop1.strongw} that $\Dt v,\Dx v\in L^2(0,T;L^2_\varphi)$, but this regularity obviously requires that
$$
\Dt h=\divv\Dt v\in L^2(0,T;H^{-1}_\phi(\Omega)).
$$
However, we do not control the $H^{-1}$-norm of $\Dt h$ since according to \eqref{2.relation}, we need to control the appropriate norm of $\Dt u_1$, where $u$ is an energy solution of \eqref{1.NStokes}. But, from equation \eqref{1.NStokes}, we know only that
$$
\Dt u\in L^2(0,T;\Cal H_\phi^{-1})+L^{4/3}(0,T;L^{4/3}_\phi(\Omega))
$$
and, as explained in Remark \ref{Rem1.dif}, this is not enough to control the reasonable norms of $\Dt u_1$. Thus, in order to avoid the assumptions on $\Dt h$, we have to use only the {\it partial} regularity, say, of the form \eqref{2.strange}. The exponent $1/3$ in \eqref{2.strange} can be replaced by $1/2-\kappa$ for every $\kappa>0$ (which is not essential for our purposes), but our method does not allow to take this exponent larger since  it utilizes property \eqref{1.kappaPi}. Finally, the condition on the integral $\int_0^t v(t)\,dt$ is added in order to be able to pose the boundary conditions.
\end{remark}
\begin{proof}[Proof of the theorem] We split the proof into several steps.
\par
{\it Step 1.} At this step we make several equivalent transformations and reduce problem \eqref{2.aug} to more standard one. First, introducing the new variable $\bar v(t):=v(t)-w(t)$, where $w(t):=(0,h(t))^t$, we obtain the divergence free problem:
\begin{equation}\label{2.aug1}
\Dt\bar v-\Dx\bar v-\Nx q=-\Dt w(t)-\Dx w(t),\ \ \divv \bar v=0,\ \ \bar v\big|_{t=0}=0.
\end{equation}
Since, obviously
$$
\|w\|_{C(0,T;W^{1,2}_\phi)}\le\|h\|_{C(0,T;W^{1,2}_\phi)},
$$
it is enough to verify estimate \eqref{2.mainest} for function $\bar v$ only. Second, we introduce the function $\tilde w(t)$ as a solution of the linear Stokes problem
$$
\Dx\tilde w(t)-\Nx q=\Dx w(t),\ \ \divv \tilde w=0,\ \ w\big|_{\partial\Omega}=0.
$$
Then, due to Proposition \ref{Prop1.stokes},
\begin{equation}\label{2.wtilde}
\|\tilde w\|_{C(0,T;\Cal V_\phi)}\le C\|\Dx w\|_{C(0,T;W^{-1,2}_\phi)}\le C_1\|h\|_{C(0,T;W^{1,2}_\phi)}
\end{equation}
and, introducing $\tilde v:=\bar v-\tilde w$, we get
\begin{equation}\label{2.aug2}
\Dt\tilde v-\Dx\tilde v-\Nx \tilde q=-\Dt(w(t)+\tilde w(t)),\ \divv\tilde v=0,\ \ \tilde v\big|_{t=0}=-\tilde w(0).
\end{equation}
Then, due to \eqref{2.wtilde}, it is enough to prove \eqref{2.mainest} for the function $\tilde v$ only. Finally, we want to get rid of the time derivative in the right-hand side of \eqref{2.aug2}. To this end, we introduce the function $V(t)$ via
$$
V(t):=\int_0^te^{-(t-s)}\tilde v(s)\,ds.
$$
Then,
$$
\Dt V(t)+V(t)=\tilde v(t)
$$
and integrating \eqref{2.aug2} in time, we arrive at
\begin{equation}\label{2.aug3}
\Dt V-\Dx V-\Nx Q=H(t),\ \ \divv V=0,\ \ V\big|_{t=0}=0,
\end{equation}
where $H(t):=\tilde w(0)+w(0)e^{-t}-\tilde w(t)-w(t)+\int_0^te^{-(t-s)}\tilde w(s)+w(s)\,ds$. Then, due to \eqref{2.wtilde},
\begin{equation}\label{2.H}
\|H\|_{C(0,T;W^{1,2}_\phi)}\le C\|h\|_{C(0,T;W^{1,2}_\phi)},
\end{equation}
where the constant $C$ is independent of $T$, and, to prove estimate \eqref{2.mainest}, it is sufficient to verify that the solutions of the linear Stokes problem \eqref{2.aug3} satisfy the following estimate:
\begin{equation}\label{2.mainest1}
\|\Dt V\|_{C(0,T;W^{1/3,2}_\phi)}+\|V\|_{C(0,T;W^{2+1/3,2}_\phi)}\le C\|H\|_{C(0,T;W^{1,2}_\phi)}
\end{equation}
for some constant $C$ independent of $T$ and the concrete choice of the weight $\phi$.
\par
{\it Step 2.} At this stage, we verify estimate \eqref{2.mainest1} for the particular case $\phi=1$ (the non-weighted case).
Namely, assuming that $H\in C(0,T;W^{1,2})$, we want to show that
\begin{equation}\label{2.mainest2}
\|\Dt V\|_{C(0,T;W^{1/3,2})}+\|V\|_{C(0,T;W^{2+1/3,2})}+\|\bar Q\|_{C(0,T;W^{1+1/3,2})}\le C\|H\|_{C(0,T;W^{1,2})}
\end{equation}
for some constant $C$ independent of $T$. Here and below $\bar Q:=Q-\Bbb S Q$.
\par
To verify \eqref{2.mainest2}, we apply the Leray projector $\Pi$ to both sides of equation \eqref{2.aug3} which gives
\begin{equation}\label{2.aug4}
\Dt V+AV=\Pi H(t),\ V\big|_{t=0}=0
\end{equation}
and, by the variation of constants formula, the solution $V$ can be written as follows:
$$
V(t)=\int_0^te^{-A(t-s)}\Pi H(s)\,ds.
$$
Then, since the Stokes operator $A$ generates an analytic semigroup (recall that it is self adjoint and positive definite in $\Cal H$), we have the estimate
$$
\|e^{At}v\|_{D(A^{\alpha+\kappa})}\le C t^{-\kappa}e^{-\alpha t}\|v\|_{D(A^\alpha)}
$$
for all $\alpha\in\R$ and $\kappa>0$, see e.g., \cite{Henry}. Then, elementary estimates give
$$
\|V\|_{C(0,T;D(A^{\alpha+1-\delta}))}\le C_\delta\|\Pi H\|_{C(0,T;D(A^{\alpha}))},
$$
where $\alpha\in\R$, $\delta>0$ and $C_\delta$ is independent of $T$. Fixing $\alpha=1/5$, $\delta=1/30$ and using the description of the fractional powers of the Stokes operator given in Section \ref{s1} as well as \eqref{1.kappaPi}, we have
$$
\|AV\|_{C(0,T;W^{1/3,2})}\le C\|\Pi H\|_{C(0,T;D(A^{1/5}))}\le C_1\|H\|_{C(0,T;W^{2/5,2})}\le C_2\|H\|_{C(0,T;W^{1,2})}.
$$
Using the maximal regularity of the Stokes operator, see Proposition \ref{Prop1.wLStokes}, and expressing $\Dt V$ from equation \eqref{2.aug4}, we get
$$
\|\Dt V\|_{C(0,T;W^{1/3,2})}+\|V\|_{C(0,T;W^{2+1/3,2})}\le C\|H\|_{C(0,T;W^{1,2})}.
$$
After that, from equation \eqref{2.aug3}, we obtain the control of the $C(0,T;W^{1/3,2})$-norm of $\Nx Q$. The Poincare inequality gives then the desired control of $\bar Q$ and proves estimate \eqref{2.mainest2}.
\par
{\it Step 3.} At this step, we deduce the weighted estimate \eqref{2.mainest1} from the non-weighted one \eqref{2.mainest2} and, thus, finish the proof of the desired estimate \eqref{2.mainest}. To this end, we introduce the function $V_\phi:=\phi V$. Then, due to assumptions \eqref{2.w},
for sufficiently small $\eb>0$, we have
\begin{equation}\label{2.equiv}
C_2\|V_\phi\|_{W^{s,2}}\le\|V\|_{W^{s,2}_\phi}\le C_1\|V_\phi\|_{W^{s,2}}
\end{equation}
for some $C_1$ and $C_2$ which are independent of $\eb$ and $s\in(0,3]$, see \cite{ZelikGlasgow}. By this reason, to prove \eqref{2.mainest1} it is sufficient to establish the analogous non-weighted estimates for function $V_\phi$. Multiplying equation \eqref{2.aug3} by $\phi$, after the elementary transformations, we have
\begin{equation}\label{2.aug41}
\Dt V_\phi-\Dx V_\phi-\phi\Nx Q=\phi H+\eb (M(x)\partial_{x_1} V_\phi+N(x) V_\phi),\ \ \divv V_\phi=\phi'V_1,\ \ V_\phi|_{t=0}=0,
\end{equation}
where $M(x):=-2\eb^{-1}\phi'\phi^{-1}$ and $N(x):=2\eb^{-1}(\phi'\phi^{-1})^2-\eb^{-1}\phi''\phi^{-1}$. Note that, due to assumptions \eqref{2.w} on the weight $\phi$,
$$
\|M(x)\|_{L^\infty}+\|N(x)\|_{L^\infty}\le C
$$
uniformly with respect to $\eb\to0$, so the last term in the right-hand side of \eqref{2.aug41} is indeed a small perturbation. To transform the term with pressure, we introduce the function
$$
Q_\phi:=\phi Q-\int_0^{x_1}\phi'(s)Q(s,x_2)\,ds
$$
Then, using that
$$
\bar Q_\phi:=Q_\phi-\Bbb S Q_\phi=\phi\bar Q,\ \ \phi\Nx Q=\Nx Q_\phi-\phi'\phi^{-1}(1,0)^t\bar Q_\phi,
$$
we transform \eqref{2.aug41} to
\begin{equation}\label{2.aug5}
\Dt V_\phi-\Dx V_\phi-\Nx Q_\phi=H_\phi,\ \ \divv V_\phi=\phi'V_1,\ \ V_\phi\big|_{t=0}=0,
\end{equation}
where $H_\phi:=\phi H+\eb (M(x)\partial_{x_1} V_\phi+N(x) V_\phi+\phi'\phi^{-1}(1,0)^t\bar Q_\phi$. Moreover, due to \eqref{2.w},
\begin{equation}\label{2.Hest}
\|H_\phi\|_{W^{1,2}}\le \|\phi H\|_{W^{1,2}}+C\eb(\|V_\phi\|_{W^{2,2}}+\|\bar Q_\phi\|_{W^{1,2}}),
\end{equation}
where the constant $C$ is independent of $\eb$. We want to apply estimate \eqref{2.mainest2} to this equation, to this end, similar to Step 1, we need to get rid of the non-zero divergence by introducing the new variables
$$
W_\phi:=(0,1)^t\phi'\Psi(V),\ \ \tilde V_\phi:=V_\phi-W_\phi.
$$
Then, it is not difficult to see
$$
\|W_\phi\|_{W^{s,2}}\le C\eb\|V_\phi\|_{W^{s-1,2}},\ \ s\in[1,3]
$$
and, therefore, for sufficiently small $\eb>0$,
$$
C_2\|\tilde V_\phi\|_{W^{s,2}}\le\|V_\phi\|_{W^{s,2}}\le C_1\|\tilde V_\phi\|_{W^{s,2}},\ \ C_2\|\Dt\tilde V_\phi\|_{W^{s,2}}\le\|\Dt V_\phi\|_{W^{s,2}}\le C_1\|\Dt\tilde V_\phi\|_{W^{s,2}},
$$
where the constants $C_1$ and $C_2$ are independent of $\eb$. On the other hand, function $\tilde V_\phi$ solves
\begin{equation}\label{2.aug6}
\Dt\tilde V_\phi-\Dx \tilde V_\phi-\Nx Q_\phi=H_\phi-\Dt W_\phi+\Dx W_\phi,\ \ \divv \tilde V_\phi=0,\ \ \tilde V_\phi\big|_{t=0}=0
\end{equation}
and, applying estimate \eqref{2.mainest2} to this equation, we get
\begin{multline}
\|\Dt V_\phi\|_{C(0,T;W^{1/3,2})}+\|V_\phi\|_{C(0,T;W^{2+1/3,2})}+\|\bar Q_\phi\|_{C(0,T;W^{1+1/3,2})}\le\\\le C(\|\Dt \tilde V_\phi\|_{C(0,T;W^{1/3,2})}+\|\tilde V_\phi\|_{C(0,T;W^{2+1/3,2})}+\|\bar Q_\phi\|_{C(0,T;W^{1+1/3,2})})\le\\\le C\|H_\phi-\Dt W_\phi+\Dx W_\phi\|_{C(0,T;W^{1,2})}\le C\|\phi H\|_{C(0,T;W^{1,2})}+\\+C\eb(\|\Dt V_\phi\|_{C(0,T;L^2)}+\|V_\phi\|_{C(0,T;W^{2,2})}+\|\bar Q_\phi\|_{C(0,T;W^{1,2})}).
\end{multline}
Thus, for sufficiently small $\eb>0$,
$$
\|\Dt V_\phi\|_{C(0,T;W^{1/3,2})}+\|V_\phi\|_{C(0,T;W^{2+1/3,2})}\le C\|\phi H\|_{C(0,T;W^{1,2})}
$$
which gives the estimate \eqref{2.mainest1} which, in turn, gives the desired estimate \eqref{2.mainest}.
\par
{\it Step 4.} Existence and uniqueness. To construct a solution of \eqref{2.aug}, we approximate the function $h\in C(0,T;W^{1,2}_\phi)$ by smooth functions $h_n$ which are convergent {\it strongly} to $h$ in that space. Let $v_n$ be the solutions of problems \eqref{2.aug} where $h$ is replaced by $h_n$ (the existence and regularity of $v_n$ follows, e.g., from Proposition \ref{Prop1.strongw}; since we do not have the problem with time derivatives for $h_n$, the existence is straightforward). Now, applying estimate \eqref{2.mainest} to the difference $v_n-v_m$, we get
$$
\|v_n-v_m\|_{C(0,T;W^{1/3,2}_\phi)}\le C\|h_n-h_m\|_{C(0,T;W^{1,2}_\phi)}
$$
and, therefore, $v_n$ is a Cauchy sequence in $C(0,T;W^{1/3,2}_\phi)$. Thus, the limit solution of \eqref{2.aug} $v\in C(0,T;W^{1/3,2}_\phi)$ and satisfies indeed estimate \eqref{2.mainest}. The uniform estimate for $\int_0^t v_n(s)\,ds$ can be obtained by integrating \eqref{2.aug} in time and using the maximal regularity \eqref{1.ws} of the Stokes operator. So, the desired solution $v(t)$ is constructed. To prove the uniqueness, assume that $v_1$ and $v_2$ are to solutions of \eqref{2.aug} with the same $h$. Then, the function $v(t):=\int_0^t(v_1(s)-v_2(s))\,ds$ is a strong solution of the Stokes problem \eqref{1.NStokes} with $u_0=g=0$. By Proposition \ref{Prop1.strongw}, $v(t)\equiv0$. Thus, the uniqueness is proved and Theorem \ref{Th2.au} is also proved.
\end{proof}
We now return to the auxiliary problem \eqref{2.au}, where $(u_1,u_2)$ solves the non-autonomous Stokes problem \eqref{1.NStokes} and $\phi$ is a weight of exponential growth rate such that $\phi\in L^2(\R)$ and satisfies \eqref{2.w} for a sufficiently small $\eb>0$. Then, the function $h(t)$ defined via \eqref{2.relation} satisfies
\begin{equation}\label{2.hu}
\|h(t)\|_{W^{1,2}_{\phi^{-1}}}\le C\eb\|\Psi(u(t))\|_{W^{1,2}_\phi}\le C\eb\|u(t)\|_{L^2_\phi}
\end{equation}
and, therefore, applying estimate \eqref{2.mainest} with the weight $\phi^{-1}$ (which also satisfies \eqref{2.w}), we see that
$$
\|v_\phi\|_{C(0,T;W^{1/3,2}_{\phi^{-1}})}\le C\eb\|u\|_{C(0,T;L^2_\phi)}.
$$
Note that, since $u\in L^\infty(0,T;L^2_b)$ and $\phi\in L^2(\R)$, the norm in the right-hand side is finite. Moreover, using the weighted Sobolev
embedding theorem $W^{1/3,2}_{\phi^{-1}}\subset L^3_{\phi^{-1}}$ (where the embedding constant is independent of $\phi$ and $\eb$, see \cite{ZelikGlasgow}), we end up with the estimate
\begin{equation}\label{2.auphi}
\|v_\phi\|_{C(0,T;L^2_{\phi^{-1}})}+\|v_\phi\|_{C(0,T;L^3_{\phi^{-1}})}\le C\eb\|u\|_{C(0,T;L^2_\phi)}.
\end{equation}
The following particular choice of the weight $\phi$ is crucial for what follows.

\begin{corollary}\label{Cor2.aumain} Let $u$ be a weak solution of equation \eqref{1.NStokes} (in the sense of Definition \ref{Def2.en} and let
$\theta_{\eb}(x)=\theta_{\eb,s}$ be the weight defined via \eqref{1.thetaweight} where $\eb>0$ is small enough and $s\in\R$ is a parameter. Then
the solution $v_\theta(t)$ of the auxiliary problem \eqref{2.au} with $\phi=\theta_{\eb,s}$ satisfies the following estimate:
\begin{equation}\label{2.aumain}
\|v_\theta\|_{C(0,T;L^2_{\theta^{-2}})}+\|v_\theta\|_{C(0,T;L^3_{\theta^{-2}})}\le C\eb\|u\|_{C(0,T;L^2_\theta)},
\end{equation}
where the constant $C$ is independent of $T$, $\eb$, $s$ and $u$.
\end{corollary}
Indeed, due to \eqref{1.thetaeb}, we may improve estimate \eqref{2.hu}:
$$
\|h(t)\|_{W^{1,2}_{\theta^{-2}}}\le C\eb\|\Psi(u(t))\|_{W^{1,2}_\theta}\le C\eb\|u\|_{L^2_\theta}
$$
and, applying estimate \eqref{2.mainest} with the weight $\phi=\theta^{-2}$ to equation \eqref{2.au}, we end up with the desired estimate \eqref{2.aumain}.
\par
We now return to energy solutions of the non-autonomous Stokes equation.

\begin{corollary}\label{Cor2.eneq} Let the weight exponential growth rate $\phi\in L^{4/3}(\R)$ and satisfy \eqref{2.w} with sufficiently small $\eb>0$. Let also $u_0\in\Cal H_b$ and $g$ satisfy \eqref{2.gas}. Then, there exists a unique energy solution $u(t)$ of the Stokes problem and this solution satisfies the estimate
\begin{equation}\label{2.useless}
\|u\|_{C(0,T;L^2_\phi)}+\|u\|_{L^2(0,T;W^{1,2}_\phi)}\le C_T\(\|u_0\|_{\Cal H_\phi}+\|g\|_{L^{4/3}(0,T;L^{4/3}_\phi)\cap L^1(0,T;L^{3/2}_\phi)}\),
\end{equation}
where the constant $C_T$ may depend on $T$, but is independent of the concrete choice of the weight. Moreover, the function $t\to\frac12\|u(t)\|^2_{L^2_\phi}-(u(t),v_\phi(t))$ is absolutely continuous and the energy identity \eqref{2.energy} holds for almost all $t\in(0,T)$.
\end{corollary}
\begin{proof} We first derive estimate \eqref{2.useless} assuming that the validity of \eqref{2.energy} is already verified. Then,
for sufficiently small $\eb>0$,
\begin{multline}\label{2.nx}
(\Nx u,\Nx(\phi^2 u))=\|\Nx u\|^2_{L^2_\phi}+2(\Nx u,\phi\phi' u)\ge\\\ge \|\Nx u\|^2_{L^2_\phi}-C\eb(\phi^2|\Nx u|,|u|)\ge\frac12\|\Nx u\|^2_{L^2_\phi}-C\eb^2\|u\|^2_{L^2_{\phi}}\ge\frac14\|u\|^2_{W^{1,2}_\phi},
\end{multline}
where we have implicitly used the weighted version of the Poincare inequality
$$
\|u\|_{L^2_\phi}\le C\|\Nx u\|_{L^2_{\phi}}.
$$
Moreover, due to \eqref{2.auphi} and the  weighted Ladyzhenskaya inequality
\begin{equation}\label{2.lady}
\|u\|_{L^4_\phi}^2\le C\|u\|_{L^2_\phi}\|u\|_{W^{1,2}_\phi},
\end{equation}
together with the H\"older inequality,
\begin{multline}\label{2.long}
|(g,\phi^2 u-v_\phi)|\le C\|g\|_{L^{4/3}_\phi}\|u\|_{L^4_\phi}+\|g\|_{L^{3/2}_\phi}\|v_\phi\|_{L^3_{\phi^{-1}}}\le C\|g\|_{L^{4/3}_\phi}\|u\|_{L^2_\phi}^{1/2}\|u\|_{W^{1,2}_\phi}^{1/2}+\\+C\|g\|_{L^{3/2}_\phi}\|u\|_{C(0,T;L^2_\phi)}\le
 C\|g\|^{4/3}_{L^{4/3}_\phi}\|u\|_{L^2_\phi}^{2/3}+C\|g\|_{L^{3/2}_\phi}\|u\|_{C(0,T;L^2_\phi)}+1/8\|u\|^2_{W^{1,2}_\phi}.
\end{multline}
Integrating now the energy identity \eqref{2.energy} in time and using \eqref{2.long}, \eqref{2.nx} and the obvious estimate
$$
|(u,v_\phi)|\le \|u\|_{L^2_{\phi}}\|v_\phi\|_{L^2_{\phi^{-1}}}\le C\eb\|u\|_{C(0,T;L^2_\phi)}^2,
$$
we arrive at
\begin{multline}
(1-C\eb)\|u\|^2_{C(0,T;L^2_\phi)}+\|u\|^2_{L^2(0,T;W^{1,2}_\phi)}\le C\|g\|^{4/3}_{L^{4/3}(0,T;L^{4/3}_\phi)}\|u\|_{C(0,T;L^2_\phi)}^{2/3}+\\+C\|g\|_{L^1(0,T;L^{3/2}_\phi)}\|u\|_{C(0,T;L^{2}_\phi)}+C\|u_0\|^2_{L^2_\phi}
\end{multline}
and estimate \eqref{2.useless} is an immediate corollary of this estimate if $\eb>0$ is small enough.
\par
Note that the uniqueness of a solution $u$ can be done exactly as in Theorem \ref{Th2.au}. To verify the existence, again similar to the proof of Theorem \ref{Th2.au}, we approximate the initial data $u_0$ by the sequence  $u_0^n\in\Cal H_\phi$ of smooth initial data
 which is convergent to $u_0$ in that space and, analogously, we approximate the external force $g$ by the sequence $g_n$ of smooth ones which is convergent in $L^{4/3}(0,T;L^{4/3}_\phi)\cap L^1(0,T;L^{3/2}_\phi)$. We note that, since the weight $\phi\in L^{4/3}(\R)$ then it is not difficult to check, $\phi\in L^{3/2}(\R)$ as well and, thanks to \eqref{1.wb}
$$
u_0\in\Cal H_\phi,\ \ g\in  L^{4/3}(0,T;L^{4/3}_\phi)\cap L^1(0,T;L^{3/2}_\phi)
$$
and, therefore, such approximations exist. Let $u_n$ be the corresponding solutions of \eqref{1.NStokes} which exist due to Proposition \ref{Prop1.strongw}. Then, applying the proved estimate \eqref{2.useless} to the differences $u_n-u_m$ of two approximation solutions (since they are smooth, the energy identity hold for them), we have
\begin{equation*}
\|u_n-u_m\|_{C(0,T;L^2_\phi)}+\|u_n-u_m\|_{L^2(0,T;W^{1,2}_\phi)}\le C_T\(\|u_0^n-u_0^m\|_{\Cal H_\phi}+\|g_n-g_m\|_{L^{4/3}(0,T;L^{4/3}_\phi)\cap L^1(0,T;L^{3/2}_\phi)}\).
\end{equation*}
Thus, $u_n-u_m$ is a Cauchy sequence in $C(0,T;L^2_\phi)\cap L^2(0,T;W^{1,2}_\phi)$ and, passing to the limit $n\to\infty$, we construct a solution $u$ of problem \eqref{1.NStokes} belonging to this space and justify estimate \eqref{2.useless}. Moreover, applying this estimate with the shifted weights $\phi(x_1-s)$, taking the supremum over $s\in\R$ and using \eqref{1.wb}, we check that $u$ belongs to the uniformly local spaces \eqref{2.defspace}. Thus, the existence of an energy solution is also verified.
\par
It only remains to prove the energy identity. To this end, we write the energy identity for $u_n$ in the equivalent integral form:
\begin{multline}\label{2.inten}
\frac12\|u_n(s)\|^2_{L^2_\phi}-(u_n(s),v^n_\phi(s))-\frac12\|u_n(\tau)\|^2_{L^2_\phi}+(u_n(\tau),v^n_\phi(\tau))=\\=
\int_\tau^s(g_n(t),\phi^2u_n(t)-v^n_\phi(t))-(\Nx u_n(t),\Nx(\phi^2 u_n(t)))\,dt,
\end{multline}
where $v_\phi^n$ are the solutions of the auxiliary problem \eqref{2.au} which correspond to the solutions $u_n$. Note that, due to estimate \eqref{2.auphi} applied to $v_\phi^n-v_{\phi}^m$, we know that $v_\phi^n$ converges strongly to $v_\phi$ in the spaces $C(0,T;L^2_{\phi^{-1}})$ and
 $C(0,T;L^3_{\phi^{-1}})$. This allows us to pass to the limit $n\to\infty$ in \eqref{2.inten} and verify that the limit function $u$ also satisfies this integral identity. Since the integral form \eqref{2.inten} of the energy identity is equivalent to the differential form \eqref{2.energy}, the energy equality is proved and the corollary is also proved.
\end{proof}
\begin{remark}\label{Rem2.rem} Note that Theorem \ref{Th2.au} does not give us the control over the $L^2(0,T;W^{1,2}_{\phi^{-1}})$-norm of the corrector $v_\phi$, so we are not allowed to multiply directly equation \eqref{1.NStokes} by $\phi^2u-v_\phi$ (the term $(\Dx u,v_\phi)$ a priori may have no sense). By this reason, we have to justify this multiplication in a different way based on the approximations and the fact that all bad terms are cancelled out since $v_\phi$ solves the {\it adjoint} equation.
\par
Mention also that the validity of the energy identity \eqref{2.energy} remains true if we replace the weight function $\phi$ by the proper {\it cut-off} function $\varphi$ with finite support (the proof just repeats word by word the one given in Corollary \ref{Cor2.eneq}). We will use this observation in the next section for verifying the uniqueness for the non-linear problem.
\end{remark}

\section{The Navier-Stokes problem: weighted energy estimates}\label{s3}
In this section, we derive the key estimate for the infinite-energy solutions of the Navier-Stokes problem in a strip $\Omega$:
\begin{equation}\label{3.eqmain}
\begin{cases}
\Dt u+(u,\Nx)u-\Dx u+\Nx p=g,\\
u\big|_{\partial\Omega}=0,\ \ \divv u=0,\ \ u\big|_{t=0}=u_0.
\end{cases}
\end{equation}
We recall that the problem possesses the mean flux first integral:
\begin{equation}\label{3.flux}
\Bbb Su_1=\frac12\int_{-1}^1u_1(t,x_1,s)\,ds=c,
\end{equation}
where the constant $c$ may depend on $t$ ($c=c(t)$), but is independent of $x_1$. At the first step, we consider the case of zero flux:
\begin{equation}\label{3.zero}
c=0.
\end{equation}
The general case will be reduced later to this particular case. Then, similar to Definition \ref{Def2.en} a function $u(t)$ is a weak (energy) solution of problem \eqref{3.eqmain} if $u$ satisfies \eqref{2.defspace} for every weight function $\phi$ of exponential growth rate such that $\phi\in L^2(\R)$ and solves \eqref{3.eqmain} in the sense of distributions. Note that, due to \eqref{2.defspace} and the Ladyzhenskaya inequality,
\begin{multline*}
\|(u,\Nx)u\|_{L^{4/3}_b((0,T)\times\Omega)}\le \|u\|_{L^4_b((0,T)\times\Omega)}\|\Nx u\|_{L^2_b((0,T)\times\Omega)}\le\\\le C\|u\|_{L^\infty(0,T;L^2_b(\Omega))}^{1/2}\|\Nx u\|_{L^2_b((0,T)\times\Omega)}^{3/2}\le C.
\end{multline*}
Moreover, due to the embedding $W^{1,2}\subset L^p$ for all $p<\infty$, we also have that
$$
\|(u,\Nx)u\|_{L^1_b(0,T;L^{3/2}_b)}\le C\|\Nx u\|_{L^2_b((0,T)\times\Omega)}\|u\|_{L^2_b(0,T;L^6_b)}\le C\|\Nx u\|^2_{L^2_b((0,T)\times\Omega)}\le C.
$$
Thus, the function $\bar g:=g-(u,\Nx)u$ satisfies assumption \eqref{2.gas} and, therefore, treating the non-linear term $(u,\Nx)u$ in equation \eqref{3.eqmain} as an external force and using Corollary \ref{Cor2.eneq}, we see that the solution $u$ of \eqref{3.eqmain} satisfies the following energy identity:
\begin{equation}\label{3.energy}
\frac d{dt}\(\frac12\|u(t)\|^2_{L^2_{\theta}}-(u(t),v_\theta(t))\)+(\Nx u(t),\Nx (\theta^2u(t)))=(g-(u(t),\Nx)u(t),\theta^2u(t)-v_\theta(t)),
\end{equation}
where the weight function $\theta=\theta_{\eb,s}$ is defined by \eqref{1.thetaweight} ($\eb>0$ is small enough and $s\in\R$) and the corrector $v_\theta$ solves the auxiliary problem \eqref{2.au}. The next theorem gives the key energy estimate for the solution $u$.

\begin{theorem}\label{Th3.0flux} Let $u(t)$, $t\in[0,T]$ be a weak solution of the Navier-Stokes problem \eqref{3.eqmain} which satisfies the zero flux condition \eqref{3.zero}. Then, the following estimate holds:
\begin{equation}\label{3.0est}
\|u\|_{L^\infty(0,T;\Cal H_b)}+\|\Nx u\|_{L^2_b((0,T)\times\Omega)}\le C(1+\|u_0\|_{\Cal H_b}+\|g\|_{L^2_b})^2,
\end{equation}
where the constant $C$ is independent of $g$, $u_0$, $T$ and $u$.
\end{theorem}
\begin{proof} To estimate the right-hand side of \eqref{3.energy}, we note that, due to the divergence free assumption,
$$
((u,\Nx)u,\theta^2u)=-\frac12(\divv u,\theta^2|u|^2)-(\theta\theta' u_1,|u|^2)=-(\theta\theta'u_1,|u|^2)
$$
and, therefore, due to \eqref{1.thetaeb} and the weighted Sobolev embedding theorem,
\begin{equation}\label{3.inest1}
|(u,\Nx u),\theta^2u)|\le C\eb\|u\|_{L^2_\theta}\|u\|^2_{L^4_{\theta}}\le C\eb\|u\|_{L^2_\theta}\|\Nx u\|^2_{L^2_\theta},
\end{equation}
where the constant $C$ is independent of $\eb$. Analogously, using also estimate \eqref{2.aumain}, we have
\begin{equation}\label{3.inest2}
|(u,\Nx)u,v_\theta)|\le C\|u\|_{L^6_\theta}\|\Nx u\|_{L^2_\theta}\|v_\theta\|_{L^3_{\theta^{-2}}}\le C\eb\|u\|_{C(0,T;L^2_\theta)}\|\Nx u\|^2_{L^2_\theta}.
\end{equation}
Inserting these estimates into the right-hand side of \eqref{3.energy}, estimating the terms involving the external force $g$ by H\"older inequality and using \eqref{2.nx} and the Poincare inequality, we end up with
\begin{equation}\label{3.wenest}
\frac d{dt}\(\frac12\|u\|^2_{L^2_{\theta}}-(u,v_\theta)\)+\alpha\|u\|^2_{L^2_\theta}+\alpha\|\Nx u\|^2_{L^2_\theta}\(1-C\eb\|u\|_{C(0,T;L^2_\theta)}\)\le C\|g\|^2_{L^2_\theta}+C\eb\|u\|_{C(0,T;L^2_\theta)}^2,
\end{equation}
where the positive constants $C$ and $\alpha$ are independent of $\eb$.
\par
We claim that \eqref{3.wenest} implies the desired estimate \eqref{3.0est}. To show that, we first make an additional assumption that the parameter $\eb>0$ is small enough to guarantee the inequality
\begin{equation}\label{3.loop}
C\eb\|u\|_{C(0,T;L^2_\theta)}\le\frac12.
\end{equation}
Then, the last term in the left hand side of \eqref{3.wenest} can be neglected and, applying the Gronwall inequality to the obtained estimate
 and using that
\begin{equation}\label{3.smallcor}
|(u,v_\theta)|\le C\|u\|_{L^2_\theta}\|v_\theta\|_{L^2_{\theta^{-1}}}\le C\eb\|u\|_{C(0,T;L^2_\theta)}^2,
\end{equation}
after the elementary transformations,
we end up with
\begin{equation}\label{3.est2}
\|u(t)\|^2_{L^2_\theta}\le C\|u_0\|_{L^2_\theta}e^{-\alpha t}+ C\|g\|^2_{L^2_\theta}+C\eb\|u\|_{C(0,T;L^2_\theta)}^2,\ \ t\in[0,T],
\end{equation}
where the constant $C$ is independent of $\eb$. For sufficiently small $\eb$, this estimate gives
\begin{equation}
\|u\|_{C(0,T;L^2_\theta)}\le C\(1+\|g\|_{L^2_\theta}+\|u_0\|_{\Cal H_\theta}\),
\end{equation}
where $C$ is independent of $\eb$. Moreover, using \eqref{1.wbeb}, we conclude that
\begin{equation}\label{3.uniform}
\|u\|_{C(0,T;L^2_\theta)}\le C_1\eb^{-1/2}\(1+\|g\|_{L^2_b}+\|u_0\|_{\Cal H_b}\),
\end{equation}
where the constant $C$ is independent of $\eb$. Now we are able to justify assumption \eqref{3.loop}. Indeed, if we fix $\eb\ll1$ in a such way that
$$
CC_1\eb^{1/2}\(1+\|g\|_{L^2_b}+\|u_0\|_{\Cal H_b}\)=\frac12,
$$
i.e.
\begin{equation}\label{3.ebfix}
\eb:=C_2\(1+\|g\|_{L^2_b}+\|u_0\|_{\Cal H_b}\)^{-2},
\end{equation}
then inequality \eqref{3.uniform} will imply the inequality \eqref{3.loop}. Since \eqref{3.loop} is satisfied for $T=0$ and the function $t\to \|u(t)\|_{L^2_\theta}$ is continuous (by the definition of an energy solution), the standard continuity arguments show that \eqref{3.loop} holds for all $T$. Thus, estimate \eqref{3.uniform} is justified if $\eb>0$ is chosen by \eqref{3.ebfix}. Integrating now \eqref{3.wenest} in time and using \eqref{3.loop} and \eqref{3.uniform}, we also get the control of the gradient:
 \begin{multline}\label{3.uniform1}
\|u\|_{C(0,T;L^2_\theta)}^2+\int_t^{t+1}\|\Nx u(s)\|^2_{L^2_\theta}\,ds\le\\\le C\eb^{-1}\(1+\|g\|_{L^2_b}+\|u_0\|_{\Cal H_b}\)^2\le C\(1+\|g\|_{L^2_b}+\|u_0\|_{\Cal H_b}\)^4.
\end{multline}
Finally, taking into the account that the weight $\theta=\theta_{\eb,s}$ depends on the parameter $s\in\R$ and that \eqref{3.uniform1} is uniform with respect to this parameter, we may deduce the desired estimate \eqref{3.0est} by taking the supremum over $s\in\R$ and using \eqref{1.wb}. Thus, the theorem is proved.
\end{proof}
Our next task is to obtain the analogue of estimate \eqref{3.0est} for the general case of non-zero flux. For simplicity, we restrict ourselves to the autonomous case where $c\in\R$ is independent of $t$ although the generalization to the time dependent fluxes $c=c(t)$ is straightforward. Following \cite{ZelikGlasgow}, we reduce the non-zero flux case to the case $c=0$ considered before by introducing the special Poiseuille type velocity profile $V_c(x)=(v_c(x_2),0)^t$ such that
\begin{equation}\label{3.c}
\Bbb Sv_c=c.
\end{equation}
Then, the difference $\bar u:=u-V_c$ will have zero flux and satisfy the perturbed version of equation \eqref{3.eqmain}
\begin{equation}\label{3.eqflux}
\Dt\bar u+(\bar u,\Nx)\bar u-\Dx\bar u+(V_c,\Nx) \bar u+(\bar u,\Nx)V_c+\Nx p=g+\Dx V_c,\ \ \divv\bar u=0,\ \ \Bbb S\bar u_1=0.
\end{equation}
Then, by definition, $u$ is an energy solution of \eqref{3.eqmain} if $\bar u$ is an energy solution of \eqref{3.eqflux}, see Definition \ref{Def2.en}.
\par
The next lemma specifies the choice of the function $V_c$.

\begin{lemma}\label{Lem3.flux} For any $c\in\R$, there exists $V_c(x)=(v_c(x_2),0)^t\in H^2(-1,1)\cap H^1_0(-1,1)$ such that \eqref{3.c} is satisfied and, for big $c$,
\begin{equation}\label{3.c.est}
1.\ \ \|v_c\|_{L^\infty}\sim c,\ \ 2.\ \ \|v_c'\|_{L^2}\sim c^{3/2},\ \ 3.\ \ \|v''_c\|_{L^{3/2}}\sim c^{7/3}.
\end{equation}
Moreover, the linearized operator $L_{c}w:=-\Dx V_c+(V_c,\Nx)w+(w,\Nx)V_c$ is energy stable, i.e., there exists $\kappa>0$ (independent of $c$) such that
\begin{equation}\label{3.positive}
(L_{c}w,w)\ge\kappa\|w\|_{H^1}^2,\ \ \forall w\in H^1_0(\Omega).
\end{equation}
\end{lemma}
Indeed, the function $v_c(x_2)$ can be found in the form
\begin{equation*}
v_c(x)=\begin{cases}
  a,\ |x|\le 1-\delta,\\
a(1-\delta^{-2}(x-1+\delta)^2),\ \ x>1-\delta,\\
a(1-\delta^{-2}(x+1-\delta)^2),\ \ x<-1+\delta,
\end{cases}
\end{equation*}
where the two parameters $a\sim c$, $\delta\sim c^{-1}$ are chosen in such way that \eqref{3.c} is satisfied. Then the straightforward calculations show that the other assumptions of the lemma are also satisfied, see \cite{ZelikGlasgow} for more details.
\par
The next theorem generalizes estimate \eqref{3.0est} for the case of the non-zero flux.

\begin{theorem}\label{Th3.main} Let $u$ be an energy solution of equation \eqref{3.eqmain} with the mean flux $c$. Then, the following estimate holds:
\begin{equation}\label{3.00est}
\|u\|_{L^\infty(0,T;L^2_b)}+\|\Nx u\|_{L^2_b((0,T)\times\Omega)}\le C(1+c^{3/2}+\|u_0\|_{L^2_b}+\|g\|_{L^2_b})^2,
\end{equation}
where the constant $C$ is independent of $u_0$, $g$, $c$, $u$ and $T$.
\end{theorem}
\begin{proof} Applying the weighted energy identity to equation \eqref{3.eqflux} (where all terms $(\bar u,\Nx)\bar u$, $(V_c,\Nx)\bar u$ and $(\bar u,\Nx)V_c$ are treated as external forces), analogously to \eqref{3.energy}, we have
\begin{multline}\label{3.enflux}
\frac {d}{dt}\(\frac12\|\bar u\|^2_{L^2_\theta}-(\bar u,v_\theta)\)+(L_c \bar u,\theta^2\bar u)=\\=(g+\Dx V_c-(\bar u,\Nx)\bar u,\theta^2\bar u-v_\theta)+((V_c,\Nx)\bar u+(\bar u,\Nx)V_c,v_\theta).
\end{multline}
Thus, we only need to estimate the extra terms appearing in this identity due to the presence of $V_c$. To do that, we assume that $\eb>0$ is small enough to satisfy
\begin{equation}\label{3.ebs}
\eb(1+c^3)\le \mu\ll1.
\end{equation}
Then, the term involving the operator $L_c$ can be estimated using \eqref{3.positive} and \eqref{3.c.est}:
\begin{multline}\label{3.lcgood}
(L_c\bar u,\theta^2\bar u)=(L_c(\theta\bar u),\theta\bar u)-([\theta']^2, |\bar u|^2)-(v_c,\theta'\theta|\bar u|^2)\ge\\\ge  \kappa\|\bar u\|^2_{W^{1,2}_\theta}-C\eb^2\|\bar u\|^2_{L^2_\theta}-Cc\eb\|\bar u\|^2_{L^2_\theta}\ge \frac\kappa2\|\bar u\|^2_{W^{1,2}_\theta}.
\end{multline}
The last term into the right-hand side of \eqref{3.enflux} can be estimated using \eqref{3.c.est} and \eqref{2.aumain}:
\begin{multline}\label{3.lcgood1}
|((V_c,\Nx)\bar u+(\bar u,\Nx)V_c,v_\theta)|\le C(\|V_c\|_{L^6_\theta}\|\Nx\bar u\|_{L^2_{\theta}}+\|\bar u\|_{L^6_\theta}\|\Nx\bar V_c\|_{L^2_{\theta}})\|v_\theta\|_{L^3_{\theta^{-2}}}\le\\\le C\eb\|V_c\|_{W^{1,2}_\theta}\|\bar u\|_{W^{1,2}_\theta}\|\bar u\|_{C(0,T;L^2_\theta)}\le C\eb^{1/2}\|V_c\|_{W^{1,2}_b}\|\bar u\|_{W^{1,2}_\theta}\|\bar u\|_{C(0,T;L^2_\theta)}\le\\\le C\eb^{1/2}(1+c^3)^{1/2}\|\bar u\|_{W^{1,2}_\theta}\|\bar u\|_{C(0,T;L^2_\theta)}\le \frac\kappa{16}\|\bar u\|^2_{W^{1,2}_\theta}+C\mu\|\bar u\|_{C(0,T;L^2_\theta)}^2.
\end{multline}
Finally, the terms involving $\Dx V_c$ can be estimated as follows:
\begin{multline}\label{3.dvc}
|(\Dx V_c,\theta^2\bar u)|=|(\Nx V_c,\Nx(\theta^2\bar u))|\le C\|\Nx V_c\|_{L^2_\theta}\|\Nx \bar u\|_{L^2_\theta}-C\eb\|\Nx V_c\|_{L^2_\theta}\|\bar u\|_{L^2_\theta}\le\\\le C\|\Nx V_c\|^2_{L^2_\theta}+\frac\kappa{16}\|\bar u\|^2_{W^{1,2}_\theta}\le C\eb^{-1}(1+c^3)+\frac\kappa{16}\|\bar u\|^2_{W^{1,2}_\theta}.
\end{multline}
and
\begin{multline}\label{3.dvc1}
|(\Dx V_c,v_\theta)|\le \|\Dx V_c\|_{L^{3/2}_\theta}\|v_\theta\|_{L^3_{\theta^{-1}}}\le C\eb\eb^{-2/3}\|\Dx V_c\|_{L^{3/2}_b}\|\bar u\|_{C(0,T;L^2_\theta)}\le\\\le C\eb^{-1/2}\mu^{5/6}(1+c^3)^{-5/6}(1+c^{7/3})\|\bar u\|_{C(0,T;L^2_\theta)}\le C\eb^{-1}+\mu\|\bar u\|_{C(0,T;L^2_\theta)}^2.
\end{multline}
Inserting estimates \eqref{3.lcgood}, \eqref{3.lcgood1}, \eqref{3.dvc} and \eqref{3.dvc1} into the identity \eqref{3.enflux} and estimating the nonlinear term by \eqref{3.inest1} and \eqref{3.inest2}, we derive the following analogue of \eqref{3.wenest}:
\begin{multline}\label{3.wenestflux}
\frac d{dt}\(\frac12\|\bar u\|^2_{L^2_\theta}-(\bar u,v_\theta)\)+\alpha\|\bar u\|^2_{L^2_\theta}+\alpha\|\bar u\|^2_{W^{1,2}_\theta}(1-C\eb\|\bar u\|_{C(0,T;L^2_\theta)})\le\\\le C_1\(\eb^{-1}(c^3+1)+\|\bar u_0\|_{L^2_\theta}^2+\|g\|^2_{L^2_\theta}\)+C_1\mu\|\bar u\|^2_{C(0,T;L^2_\theta)},
\end{multline}
where the positive constants $C$, $C_1$, $\alpha$ and $\mu\ll1$ are independent of $\eb$ and $T$. The rest of the proof repeats word by word the end of the proof of Theorem \ref{Th3.0flux}. Indeed, under the extra assumption that
\begin{equation}\label{3.extra3}
C\eb\|\bar u\|_{C(0,T;L^2_\theta)}\le\frac12,
\end{equation}
estimate \eqref{3.wenestflux} implies that
\begin{equation}\label{3.est123}
\|\bar u\|_{C(0,T;L^2_\theta)}\le C_2\eb^{-1/2}\(1+c^{3/2}+\|u_0\|_{L^2_b}+\|g\|_{L^2_b}\),
\end{equation}
see the derivation of \eqref{3.uniform}. Thus, if we fix
\begin{equation}\label{3.ebfin}
\eb:=\mu\(1+c^{3/2}+\|u_0\|_{L^2_b}+\|g\|_{L^2_b}\)^{-2},
\end{equation}
where $\mu>0$ is small enough, then both assumptions \eqref{3.ebs} and \eqref{3.extra3} will be satisfied and, therefore, \eqref{3.est123} is justified. Then, integrating \eqref{3.wenestflux} in time and using \eqref{3.est123} together with \eqref{3.ebfin}, we end up with the analogue of \eqref{3.uniform1}:
 \begin{multline}\label{3.uniform2}
\|\bar u\|_{C(0,T;L^2_\theta)}^2+\int_t^{t+1}\|\Nx \bar u(s)\|^2_{L^2_\theta}\,ds\le\\\le C\eb^{-1}\(1+c^{3/2}+\|g\|_{L^2_b}+\|u_0\|_{L^2_b}\)^2\le C\(1+c^{3/2}+\|g\|_{L^2_b}+\|u_0\|_{L^2b}\)^4.
\end{multline}
which implies the desired estimate \eqref{3.00est} and finishes the proof of the theorem.
\end{proof}

\section{The Navier-Stokes problem: existence, uniqueness and regularity of solutions}\label{s4}
In this section, we show the well-poseness of the Navier-Stokes problem \eqref{3.eqmain} in the uniformly local spaces. We start with the uniqueness result.

\begin{theorem}\label{Th4.unique} Let $u^{(1)}$ and $u^{(2)}$ be two energy solutions of problem \eqref{3.eqmain} which satisfy \eqref{3.flux} with the same constant $c$. Then, the following estimate holds:
\begin{equation}\label{4.unloc}
\|u^{(1)}-u^{(2)}\|_{L^\infty(0,T;L^2_b)}+\|\Nx u^{(1)}-\Nx u^{(2)}\|_{L^2_b((0,T)\times\Omega)}\le C_T\|u^{(1)}(0)-u^{(2)}(0)\|_{L^2_b},
\end{equation}
where the constant $C_T$ depends only on $T$ and on the uniformly local energy norms of $u^{(1)}$ and $u^{(2)}$.
\end{theorem}
\begin{proof} Let $u(t)=u^{(1)}(t)-u^{(2)}(t)$. Then, this function solves
\begin{equation}\label{4.dif}
\Dt u-\Dx u+\Nx p+(u^{(2)},\Nx)u+(u,\Nx)u^{(1)}=0,\ \divv u=0,\ u\big|_{t=0}=u^{(1)}(0)-u^{(2)}(0),\ \Bbb Su_1=0.
\end{equation}
Let us also introduce the cut-off function $\psi\in C^\infty_0(\R)$ such that
$$
1.\ \ \psi(x_1)=1,\ \ x_1\in(0,1),\ \ 2. \ \ \psi(x_1)=0,\ \ x_1\notin(-1,2)
$$
and let $\psi_s(x):=\psi(x-s)$. Moreover, we assume that this cut off function depends on a small parameter $\mu>0$ in such way that
$$
3. \ \ (\psi^2)''\le \mu,\ \ \ 4.\ \ \|\psi\|_{L^\infty}+\|\psi\psi'\|_{L^\infty}\le C,
$$
where $C$ is independent of $\mu$. It is not difficult to see that such a cut off function exists, however, the $L^\infty$-norm of its derivative must grow as $\mu\to0$:
$$
\|\psi'\|_{L^\infty}\le C_\mu.
$$
We write down the energy identity \eqref{2.energy} with the weight $\phi$ replaced by the cut off function $\psi_s$, see Remark \ref{Rem2.rem}:
\begin{equation}\label{4.cut}
\frac d{dt}\(\frac12\|u\|^2_{L^2_{\psi_s}}-(v_\psi,u)\)+(\Nx u,\Nx(\psi_s^2u))=-((u^{(2)},\Nx)u+(u,\Nx)u^{(1)},\psi_s^2u-v_\psi),
\end{equation}
where the corrector $v_\psi$ solves \eqref{2.au} (where $\phi$ is replaced by $\psi_s$) and, since, obviously,
$$
|\psi_s'(x)\psi_s(x)|\le C\theta_{\eb,s}(x)^2
$$
uniformly with respect to $s$,
 due to Theorem \ref{Th2.au}, the corrector  $v_\psi$ satisfies the analogue of
\eqref{2.aumain}:
\begin{equation}\label{4.aucut}
\|v_\psi\|_{L^2(0,T;L^2_{\theta^{-2}_{\eb,s}})}+\|v_\psi\|_{L^2(0,T;L^3_{\theta^{-2}_{\eb,s}})}\le C\|u\|_{C(0,T;L^2_{\theta_{\eb,s}})},
\end{equation}
where $\eb>0$ is small enough and the constant $C$ is independent of $s\in\R$. To simplify the notations, we will write below $\theta$ and $\psi$ instead of $\theta_{\eb,s}$ and $\psi_s$.
\par
Our task now is to estimate every term in \eqref{4.cut}. First, using that $\psi_s$ is identically zero outside of $(s-1,s+2)$ and denoting $\Omega_{s-1,s+2}:=(s-1,s+2)\times\Omega$, we have
\begin{multline}\label{4.nxpsi}
(\Nx u,\Nx(\psi^2 u))\ge\|\Nx(\psi u)\|^2_{L^2}-C\|\Nx(\psi u)\|_{L^2}\|\psi'u\|_{L^2}-C\|\psi'u\|^2_{L^2}\ge \frac12\|\Nx(\psi u)\|^2_{L^2}-\\-
C_\mu\|u\|^2_{L^2(\Omega_{s-1,s+2})}\ge \frac12\|\Nx u\|^2_{L^2(\Omega_s)}-C_\mu\|u\|^2_{L^2(\Omega_{s-1,s+2})}.
\end{multline}
Next, using the Ladyzhenskaya and H\"older inequality,
\begin{multline}\label{4.non1}
|(u,\Nx) u^{(1)},\psi^2 u)|\le \|\Nx u^{(1)}\|_{L^2(\Omega_{s-1,s+2})}\|\psi u\|^2_{L^4}\le\\\le C\|\Nx u^{(1)}\|_{L^2(\Omega_{s-1,s+2})}\|\psi u\|_{L^2}\|\Nx(\psi u)\|_{L^2}\le\frac1{16}\|\Nx(\psi u)\|^2_{L^2}+C\|\Nx u^{(1)}\|^2_{L^2(\Omega_{s-1,s+2})}\|\psi u\|^2_{L^2}.
\end{multline}
Integrating by parts and arguing analogously, we also have
\begin{equation}\label{4.non2}
|(u^{(2)},\Nx)u,\psi^2 u)|\le C(|u^{(2)}|,|\psi u|\cdot|\psi'u|)
\end{equation}
and, due to the H\"older inequality,
\begin{multline}
C(|u^{(2)}|,|\psi u|\cdot|\psi'u|)\le C\|u^{(2)}\|_{L^4(\Omega_{s-1,s+2})}\|\psi u\|_{L^4}\|\psi'u\|_{L^2}\le\\\le C_\mu\|u^{(2)}\|_{W^{1,2}(\Omega_{s-1,s+2})}^{1/2}\|\psi u\|_{L^2}^{1/2}\|\Nx(\psi u)\|_{L^2}^{1/2}\|u\|_{L^2(\Omega_{s-1,s+2})}\le \frac1{16}\|\Nx(\psi u)\|^2_{L^2}+\\+C\|\Nx u^{(2)}\|^2_{L^2(\Omega_{s-1,s+2})}\|\psi u\|^2_{L^2}+C_\mu\|u\|^2_{L^2(\Omega_{s-1,s+2})},
\end{multline}
where we have implicitly used the Ladyzhenskaya inequality together with the fact that the uniformly local  $L^2$-norm of $u^{(2)}$ is under the control due to the energy estimate. Combining the obtained estimates, we get
\begin{multline}\label{4.non-main}
|((u^{(2)},\Nx)u+(u,\Nx)u^{(1)},\psi_s^2u)|\le \frac3{16}\|\Nx(\psi u)\|^2_{L^2}+\\+C(\|\Nx u^{(1)}\|^2_{L^2(\Omega_{s-1,s+2})}+\|\Nx u^2\|^2_{L^2(\Omega_{s-1,s+2})})\|\psi u\|^2_{L^2}+C_\mu\|u\|^2_{L^2(\Omega_{s-1,s+2})},
\end{multline}
where the constant $C$ is independent of $s$ and $T$ and the constant $C_\mu$ depends only on the small parameter $\mu$ involved in the weight function $\psi$. It only remains to estimate the terms containing the corrector $v_\psi$. Using the weighted interpolation and  H\"older inequalities together with estimate \eqref{4.aucut}, we have
\begin{multline}\label{4.non-main1}
|(u,\Nx)u^{(1)},v_\psi)|\le C(\theta|u|\cdot\theta|\Nx u^{(1)}|,\theta^{-2}|v_\psi|)\le\\ \|u\|_{L^6_{\theta}}\|\Nx u^{(1)}\|_{L^2_{\theta}}\|v_\psi\|_{L^3_{\theta^{-2}}}\le \mu\|\Nx u^{(1)}\|^2_{L^2_\theta}\|u\|_{C(0,T;L^2_\theta)}^2+C_\mu\|u\|^2_{L^6_\theta}\le\\\le \mu\|\Nx u^{(1)}\|^2_{L^2_\theta}\|u\|_{C(0,T;L^2_\theta)}^2+\mu\|\Nx u\|^2_{L^2_\theta}+C_\mu\|u\|^2_{L^2_\theta},
\end{multline}
\begin{multline}\label{4.non-main2}
|(u^{(2)},\Nx)u,v_\psi)|\le\|u^{(2)}\|_{L^6_\theta}\|\Nx u\|_{L^2_\theta}\|v_\psi\|_{L^2_{\theta^{-2}}}\le\mu\|\Nx u\|^2_{L^2_\theta}+\\+C_\mu\|u^{(2)}\|_{L^2_{\theta}}^{2/3}\|\Nx u^{(2)}\|_{L^2_\theta}^{4/3}\|u\|_{C(0,T;L^2_\theta)}^2\le \mu \|\Nx u\|^2_{L^2_\theta}+\mu\|\Nx u^{(2)}\|^2_{L^2_\theta}\|u\|_{C(0,T;L^2_\theta)}^2+C_\mu\|u\|^2_{C(0,T;L^2_{\theta})}.
\end{multline}
Integrating now the energy identity \eqref{4.cut} in time $t\in[0,T]$ and using \eqref{4.non-main}, \eqref{4.non-main1},\eqref{4.non-main2} and \eqref{4.aucut} together with \eqref{4.nxpsi}, we end up with
\begin{multline}
\frac12\|\psi u(T)\|^2_{L^2}-(u(T),v_\psi(T))+\alpha \int_0^T\|\Nx u(t)\|_{L^2(\Omega_s)}^2\,dt\le\\\le C\int_0^T(\|\Nx u^{(1)}(t)\|^2_{L^2(\Omega_{s-1,s+2})}+\|\Nx u^{(2)}(t)\|^2_{L^2(\Omega_{s-1,s+2})})\|\psi u(t)\|^2_{L^2}\,dt+\\+\mu \int_0^T\|\Nx u(t)\|^2_{L^2_\theta}\,dt+C_\mu\int_0^T\|u(t)\|^2_{L^2_\theta}\,dt+\\+C\mu\int_0^T(\|\Nx u^{(1)}(t)\|^2_{L^2_\theta}+\|\Nx u^{(2)}(t)\|^2_{L^2_\theta})\,\|u\|^2_{C(0,T;L^2_\theta)}dt+\\+C_\mu T\|u\|^2_{C(0,T;L^2_\theta)}+C_\mu\|\psi u(0)\|^2_{L^2}.
\end{multline}
Now, using that $v_\psi(t)=2\psi\psi'(0,1)^t\Psi_u(T)$, where $\Psi_u$ is a stream function of $u$ and our assumptions
\begin{multline}\label{4.new}
(u(T),v_\psi(T))=(u_2(T),(\psi^2)'\Psi_u(t))=-(\partial_{x_1}\Psi_u(T),(\psi^2)'\Psi_u(T))=\\=(|\Psi_u(T)|^2,(\psi^2)'')\le C\mu\|\Psi_u(T)\|_{L^2(\Omega_{s-1,s+2})}^2\le C\mu\|u\|^2_{C(0,T;L^2_\theta)},
\end{multline}
where we have implicitly used the obvious inequality
$$
\|u\|_{L^2(\Omega_{s-1,s+2})}\le C\|u\|^2_{L^2_{\theta_s}}.
$$
Using now the fact that the $L^2(0,T;L^2_\theta)$-norms of $\Nx u^{(1)}$ and $\Nx u^{(2)}$ are bounded for energy solutions, we see that, say, for $T\le1$,
\begin{multline}
\|\psi u(T)\|^2_{L^2}+\alpha \int_0^T\|\Nx u(t)\|_{L^2(\Omega_s)}^2\,dt\le C\int_0^T(\|\Nx u^{(1)}(t)\|^2_{L^2_\theta}+\|\Nx u^{(2)}(t)\|^2_{L^2_\theta})\|\psi u(t)\|^2_{L^2}\,dt+\\+\mu \int_0^T\|\Nx u(t)\|^2_{L^2_\theta}\,dt+C_\mu\int_0^T\|u(t)\|^2_{L^2_\theta}\,dt+
(C\mu+C_\mu T)\|u\|^2_{C(0,T;L^2_\theta)}+C_\mu\|\psi u(0)\|^2_{L^2}.
\end{multline}
We are now ready to apply the Gronwall inequality with respect to function $T\to\|\psi u(T)\|^2_{L^2_\theta}$. Then, using that the  $L^2(0,T;L^2_\theta)$-norms of $\Nx u^{(1)}$ and $\Nx u^{(2)}$ are bounded, we arrive at
\begin{multline}
\|\psi_s u(T)\|^2_{L^2}+\alpha \int_0^T\|\Nx u(t)\|_{L^2(\Omega_s)}^2\,dt\le C\mu \int_0^T\|\Nx u(t)\|^2_{L^2_{\theta_{\eb,s}}}\,dt+\\+C_\mu\int_0^T\|u(t)\|^2_{L^2_{\theta_{\eb,s}}}\,dt+
(C\mu+C_\mu T)\|u\|^2_{C(0,T;L^2_{\theta_{\eb,s}})}+C_\mu\|\psi_s u(0)\|^2_{L^2}.
\end{multline}
Now, we multiply this inequality by $\theta_{\eb,\tau}(s)$, where $\tau\in\R$ is a parameter, and integrate over $s\in\R$. Then, using \eqref{1.wequiv} together with the obvious inequality
\begin{equation}\label{4.obvious}
\int_{\R}\theta_{\eb,\tau}\theta_{\eb,s}(x)\,ds\le C_\eb\theta_{\eb,\tau}(x),
\end{equation}
see e.g., \cite{ZelJMFM}, we end up with
\begin{multline}
\|u(T)\|^2_{L^2_{\theta_{\eb,\tau}}}+C(\alpha-C\mu) \int_0^T\|\Nx u(t)\|_{L^2_{\theta_{\eb,\tau}}}^2\,dt\le\\\le C_\mu\int_0^T\|u(t)\|^2_{L^2_{\theta_{\eb,\tau}}}\,dt+
(C\mu+C_\mu T)Y_\tau(u,T)+C_\mu\|u(0)\|^2_{L^2_{\theta_{\eb,\tau}}},
\end{multline}
where $Y_\tau(u,T):=\int_{s\in\R}\theta_{\eb,\tau}(s)\|u\|^2_{C(0,T;L^2_{\theta_{\eb,s}})}$. Note that, in contrast to the other term, we cannot simplify $Y_{\tau}(u,T)$ since we cannot change the order of integration and supremum. Assuming that $\mu$ is small enough, we have
\begin{multline}
\|u(T)\|^2_{L^2_{\theta_{\eb,\tau}}}+\alpha_1 \int_0^T\|\Nx u(t)\|_{L^2_{\theta_{\eb,\tau}}}^2\,dt\le\\\le C_\mu\int_0^T\|u(t)\|^2_{L^2_{\theta_{\eb,\tau}}}\,dt+
(C\mu+C_\mu T)Y_\tau(u,T)+C_\mu\|u(0)\|^2_{L^2_{\theta_{\eb,\tau}}}.
\end{multline}
We apply once more the Gronwall inequality, now with respect to function $T\to \|u(T)\|^2_{L^2_{\theta_{\eb,\tau}}}$. Then, assuming that $T$ is small enough that $C_\mu T\ll1$, we arrive at
\begin{equation}\label{4.good}
\|u\|^2_{C(0,T;L^2_{\theta_{\eb,\tau}})}+\alpha_2 \int_0^T\|\Nx u(t)\|_{L^2_{\theta_{\eb,\tau}}}^2\,dt\le
(C\mu+C_\mu T)Y_\tau(u,T)+C_\mu\|u(0)\|^2_{L^2_{\theta_{\eb,\tau}}}.
\end{equation}
Multiplying this inequality by $\theta_{\eb,\tau_1}(\tau)$ and using \eqref{4.obvious} again, we get
\begin{equation}\label{4.finest}
Y_{\tau_1}(u,T)\le (C\mu+C_\mu T)Y_{\tau_1}(u,T)+C_\mu\|u(0)\|^2_{L^2_{\theta_{\eb,\tau_1}}}.
\end{equation}
Finally, fixing $\mu$ and $T=T(\mu)$ being so small that $C\mu+C_\mu T<\frac12$, we see that
$$
Y_{\tau_1}(u,T)\le C\|u(0)\|^2_{L^2_{\theta_{\eb,\tau_1}}}.
$$
Inserting this estimate into the right-hand side of \eqref{4.good}, we prove that
\begin{equation}\label{4.verygood}
\|u\|^2_{C(0,T;L^2_{\theta_{\eb,\tau}})}+\alpha_2 \int_0^T\|\Nx u(t)\|_{L^2_{\theta_{\eb,\tau}}}^2\,dt\le
C\|u(0)\|^2_{L^2_{\theta_{\eb,\tau}}}
\end{equation}
holds for small $T$. Iterating this estimate, we verify that it actually holds for all $T\ge0$ (of course, with the constant $C$ depending on $T$).
Taking the supremum with respect to $\tau\in\R$ from both sides of this inequality and using \eqref{1.wb}, we get the desired estimate \eqref{4.unloc} and finish the proof of the theorem.
\end{proof}
\begin{remark}\label{Rem4.lipext} Arguing in a similar way, one shows that if $u^{(1)}$ and $u^{(2)}$ are two solutions of the Navier-Stokes equation which correspond to different non-autonomous external forces $g_1(t)$ and $g_2(t)$, then the following analogue of estimate \eqref{4.verygood} holds for $u=u^{(1)}-u^{(2)}$:
\begin{equation}\label{4.verygood1}
\|u\|^2_{C(0,T;L^2_{\theta_{\eb,\tau}})}+\alpha_2 \int_0^T\|\Nx u(t)\|_{L^2_{\theta_{\eb,\tau}}}^2\,dt\le
C(\|u^{(1)}(0)-u^{(2)}(0)\|^2_{L^2_{\theta_{\eb,\tau}}}+\|g_1-g_2\|_{L^2(0,T;L^2_{\theta_{\eb,\tau}})}^2),
\end{equation}
where the constant $C$ depends only on the uniformly local energy norms of $u^{(1)}$ and $u^{(2)}$. Indeed, to derive this estimate, one only need to estimate the additional term $(g_1-g_2,\psi^2 u-v_\psi)$ which is done in a standard way with the help of Cauchy Schwartz inequality and estimate \eqref{4.aucut}.
\end{remark}
We are now ready to prove the existence of a solution for the Navier-Stokes problem \eqref{3.eqmain}.
\begin{proposition}\label{Prop4.exist} For every $c\in\R$ and initial data $u_0\in V_c+\Cal H_b$, there exist a unique global solution $u(t)$ of problem \eqref{3.eqmain} with the mean flux $c$ and, therefore, the solution semigroup
\begin{equation}
S(t):V_c+\Cal H_b\to V_c+\Cal H_b,\ \ S(t)u_0\to u(t),\ \ t\ge0
\end{equation}
is well-defined.
\end{proposition}
\begin{proof} Since the uniqueness is already verified, we only need to prove the existence. For simplicity, we assume that $c=0$ (the general case is reduce to this particular one exactly as in Theorem \ref{Th3.main}). We approximate
 the initial data $u_0\in\Cal H_b$ and the external forces $g\in [L^2_b(\Omega)]^2$ by the sequences $u_0^n\in \Cal H$ and $g_n\in [L^2(\Omega)]^2$ in such way that
\begin{equation}\label{4.approximation}
1.\ g_n\to g\ \text{and}\
 u_0^n\to u_0 \ \text{in}\ [L^2_{loc}(\bar\Omega)]^2,\ \ 2.\ \ \|u_n\|_{\Cal H_b}\le C\|u_0\|_{\Cal H_b} \ \text{and} \ \|g_n\|_{L^2_b}\le \|g\|_{L^2_b}.
\end{equation}
Obviously such approximations exist. We denote by $u_n(t)$ the solutions of the Navier-Stokes problem \eqref{3.eqmain} where the initial data $u_0$ and the external force $g_n$ are replaced by $u_0^n$ and $g_n$ respectively. Since $u_0^n$ and $g_n$ are square integrable, $u_n(t)$ is a usual energy solution of the Navier-Stokes problem and the existence of such solutions is well-known, see \cite{temam,temam1} or \cite{babin1}.
Moreover, due to estimate \eqref{3.00est} and the second assertion of \eqref{4.approximation}, we have
\begin{equation}\label{4.goodappr}
\|u_n\|_{C(0,T;L^2_b)}+\|\Nx u_n\|_{L^2_b((0,T)\times\Omega)}\le C(1+\|u_0\|_{\Cal H_b}+\|g\|_{L^2_b}),
\end{equation}
where $C$ is independent of $n$. Let us now fix some weight $\theta_{\eb,\tau}$ with sufficiently small $\eb>0$. Then, \eqref{4.approximation} implies that $u_0^n\to u_0$ and $g_n\to g$ strongly in $L^2_{\theta_{\eb,\tau}}$, see e.g., \cite{ZelikCPAM} and consequently, estimate \eqref{4.verygood1} applied to solutions $u_n$ and $u_m$ shows that $u_n$ is a Cauchy sequence in $C(0,T;L^2_\theta)\cap L^2(0,T;W^{1,2}_\theta)$.
\par
Let $u(t)$ be the limit of this sequence. Then, due to \eqref{4.goodappr}, $u\in L^\infty(0,T;L^2_b)$ and $\Nx u\in L^2_b((0,T)\times\Omega)$. Moreover, passing to the limit $n\to\infty$ in the equations for $u_n(t)$, we see that $u(t)$ solves the Navier-Stokes problem \eqref{3.eqmain} (the passage to the limit in the nonlinear terms is straightforward since we have the {\it strong} convergence in $C(0,T;L^2_{\theta})$ and in $L^2(0,T;W^{1,2}_\theta)$. Thus, the desired solution $u(t)$ is constructed and the proposition is proved.
\end{proof}
We conclude this section by stating the preliminary result on the higher regularity of solutions which will be improved in the next section.
\begin{proposition}\label{Prop4.badreg} Let the assumptions of Theorem \ref{Th3.main} hold and let, in addition, the initial data $u_0\in [W^{2,2}_b(\Omega)]^2$. Then, the solution $u(t)\in [W^{2,2}_b(\Omega)]^2$ for all $t\ge0$ and the following estimate holds:
\begin{equation}\label{4.badest}
\|u\|_{L^\infty(0,T;W^{2,2}_b)}\le Q_T(\|u_0\|_{W^{2,2}_b})+Q_T(\|g\|_{L^2_b}),
\end{equation}
where the function $Q_T$ depends on $T$, but is independent of $u_0$ and $g$.
\end{proposition}
\begin{proof}
 We give only the formal derivation of estimate \eqref{4.badest} which can be justified in a standard way (e.g., by approximating the solution $u$ by the square integrable solutions which regularity is well known) and also consider for simplicity only the case $c=0$. We differentiate equation \eqref{3.eqmain} in time and denote $v=\Dt u$. Then this function solves the equation
\begin{equation}\label{4.tdif}
\Dt v-\Dx v+(v,\Nx)u+(u,\Nx)v+\Nx p=0,\ \ v\big|_{t=0}=\Dt u(0),\ \ \divv v=0.
\end{equation}
This equation has the form of \eqref{4.dif}, so repeating word by word the arguments given in the proof of Theorem \ref{Th4.unique}, we derive that
\begin{equation}\label{4.derest}
\|v\|_{C(0,T;L^2_\theta)}+\|\Nx v\|_{L^2(0,T;L^2_\theta)}\le C_T\|\Dt u(0)\|_{L^2_\theta},
\end{equation}
where $\theta_{\eb,\tau}$ is a weight function, $\eb\ll1$ and $\tau\in\R$. Now, applying the Leray projector $P$ to equation \eqref{3.eqmain} we have
\begin{equation}\label{4.time}
\Dt u=\Pi\Dx u-\Pi[(u,\Nx)u]+\Pi g.
\end{equation}
Using the regularity of the Leray projector in $L^2_b$ and the embedding $W^{2,2}_b\subset C_b$, we see that
$$
\|\Dt u(0)\|_{L^2_b}\le C\|u_0\|_{W^{2,2}_b}(1+\|u_0\|_{W^{2,2}_b}+C\|g\|_{L^2_b})
$$
and taking the supremum with respect to $\tau\in\R$ from both sides of \eqref{4.derest}, we have
\begin{equation}\label{4.badtest}
\|\Dt u(T)\|_{L^2_b}\le C_T\|u_0\|_{W^{2,2}_b}(1+\|u_0\|_{W^{2,2}_b}+C_T\|g\|_{L^2_b}).
\end{equation}
To derive the desired estimate \eqref{4.badest} from \eqref{4.badtest}, we need the following lemma which also has an independent interest.
\begin{lemma}\label{Lem4.auto} Let $g\in [L^2_b(\Omega)]^2$ and $u\in\Cal V_b$ be the solution of the stationary Navier-Stokes problem
\begin{equation}\label{4.statNS}
-\Dx u+(u,\Nx)u+\Nx p=g,\ \ \divv u=0, u\big|_{\partial\Omega}=0.
\end{equation}
Then, $u\in [W^{1,2}_b(\Omega)]^2$ and the following estimate holds:
\begin{equation}\label{4.statreg}
\|u\|_{W^{2,2}_b}\le Q(\|g\|_{L^2_b})+Q(\|u\|_{L^2_b})
\end{equation}
for some monotone function which is independent of $u$ and $g$,
\end{lemma}
\begin{proof}[Proof of the lemma] Indeed, since the stationary solution $u$ formally satisfies the non-autonomous equation \eqref{3.eqmain}, estimate \eqref{3.0est} is formally applicable to it and gives that
\begin{equation}
\|\Nx u\|_{L^2_b}\le C(1+\|g\|^2_{L^2_b}+\|u\|^2_{L^2_b})^2.
\end{equation}
Moreover,
\begin{equation}
\|(u,\Nx)u\|_{L^2_b}\le C\|u\|_{L^\infty}\|\Nx u\|_{L^2_b}\le C\|\Nx u\|_{L^2_b}^{3/2}\|u\|_{W^{2,2}_b}^{1/2}\le \mu\|u\|_{W^{2,2}_b}+C_\mu\|\Nx u\|^3_{L^2_b},
\end{equation}
where $\mu>0$ can be arbitrarily small. Using now the maximal regularity of the Stokes operator in $L^2_b$ and interpreting the nonlinear term in \eqref{4.statNS} as a part of the external forces, we derive the desired estimate \eqref{4.statreg} and finish the proof of the lemma.
\end{proof}
Now it is not difficult to finish the proof of the proposition. Indeed, interpreting the Navier-Stokes equation \eqref{3.eqmain} as stationary equation \eqref{4.statNS} at every fixed point $t$ with the external forces $g(t)=g-\Dt u(t)$ and using estimate \eqref{4.statreg} together with \eqref{3.0est} and \eqref{4.badtest}, we end up with \eqref{4.badest} and finish the proof of the proposition.
\end{proof}

\section{Dissipativity and smoothing property}\label{s5}
In this concluding section, we prove the dissipative estimate for the infinite-energy solutions of the Navier-Stokes problem \eqref{3.eqmain} as well as the so-called smoothing property. These properties are crucial, e.g., for the attractor theory, see e.g., \cite{MirZel} and references therein. We start with the dissipativity. Next theorem refines the result of Theorem \ref{Th3.main}.

\begin{theorem}\label{Th5.dis} Let the assumptions of Theorem \ref{Th3.main} hold. Then, the energy solution $u(t)$ satisfies the following estimate:
\begin{equation}\label{5.dis}
\|u(t)\|_{L^2_b}\le Q(\|u_0\|_{L^2_b})e^{-\alpha t}+C(1+c^{3/2}+\|g\|_{L^2_b}^2),
\end{equation}
where the positive constants $\alpha$ and $C$ and the monotone function $Q$ are independent of $t$, $c$ and $u_0$.
\end{theorem}
\begin{proof} The proof given below follows \cite{ZelikGlasgow}, see \cite{ZelPenCPAA,ZelJMFM} for the alternative method based on using the time dependent weights $\theta_{\eb(t),s}$. We start with estimate \eqref{3.wenestflux} which is proved under  assumption \eqref{3.ebs} on the parameter $\eb$. Applying the Gronwall estimate to it assuming that \eqref{3.extra3} holds, analogously to \eqref{3.est2}, we end up with
$$
\|u(t)\|^2_{L^2_{\theta_{\eb,s}}}\le C_1\|u_0\|^2_{L^2_{\theta_{\eb,s}}}e^{-\alpha t}+C_1\eb^{-1}(1+c^{3/2}+\|g\|_{L^2_b}^2)^2+C_1\mu\|u\|_{C(0,T;L^2_{\theta_{\eb,s}})}^2,
$$
where the constants $C_1, \alpha>0$ are independent of $T$ and $s\in\R$.
Taking the supremum over $t\in[0,T]$ from both sides of this inequality and utilizing that $\mu\ll1$, we obtain the estimate for $\|u\|_{C(0,T;L^2_{\theta_{\eb,s}})}$ and, inserting it into the right-hand side of the above derived inequality, we get
\begin{equation}\label{5.estmain}
\|u(t)\|^2_{L^2_{\theta_{\eb,s}}}\le C_1(\mu+e^{-\alpha t})\|u_0\|^2_{L^2_{\theta_{\eb,s}}}+C_1\eb^{-1}K,\ \ K:=(1+c^{3/2}+\|g\|_{L^2_b})^2.
\end{equation}
This estimate is the main technical tool to verify the dissipativity. We recall that, for its validity, the parameter $\eb$
should satisfy two inequalities: \eqref{3.ebs} and \eqref{3.extra3}. Actually, we will use in the sequel only such $\eb$ which satisfy
\begin{equation}\label{5.goodebs}
\eb\le \eb_0:=\mu\(1+c^{3/2}+\|g\|_{L^2_b}\)^2,
\end{equation}
where $\mu\ll1$ is fixed, so \eqref{3.ebs} is automatically satisfied. To verify \eqref{3.extra3}, it is sufficient to note that, since the right-hand side of \eqref{5.estmain} is monotone decreasing in time, so \eqref{3.extra3} will be formally satisfied if
$$
\eb^2(C_1(\mu+1)\|u_0\|^2_{L^2_{\theta_{\eb,s}}}+C_1\eb^{-1}K)\le\frac1{4C^2}
$$
and this can be justified using the standard continuity arguments. Finally, using \eqref{5.goodebs} and the fact that $\mu\ll1$, we see that the last inequality will be satisfied for all $t\ge0$ if
\begin{equation}\label{5.goodind}
\eb^2\|u_0\|^2_{L^2_{\theta_{\eb,s}}}\le \frac1{8C^2C_1}.
\end{equation}
Thus, estimate \eqref{5.estmain} holds for all $t\ge0$ if $\eb$ and $u_0$ satisfies inequalities \eqref{5.goodebs} and \eqref{5.goodind}.
\par
At the next step, we get rid of the "non-dissipative" parameter $\mu$ in the right-hand side of \eqref{5.estmain}. To this end, we iterate it with a sufficiently large timestep $T$:
$$
\|u(nT)\|_{L^2_{\theta_{\eb,s}}}^2\le C_1(\mu+e^{-\alpha T})\|u((n-1)T)\|^2_{L^2_{\theta_{\eb,s}}}+C_1\eb^{-1}K
$$
which after the summations of the geometric progressions gives
$$
\|u(nT)\|_{L^2_{\theta_{\eb,s}}}^2\le C_2e^{-\beta n T}\|u_0\|^2_{L^2_{\theta_{\eb,s}}}+C_2\eb^{-1}K
$$
for the appropriate positive constants $\beta$ and $C_2$ which are independent of $n$. Combining this estimate with \eqref{5.estmain} (in order to estimate $u(t)$ for $t\in((n-1)T,nT)$), we end up with the desired estimate:
\begin{equation}\label{5.bestest}
\|u(t)\|^2_{L^2_{\theta_{\eb,s}}}\le C_3e^{-\gamma t}\|u_0\|^2_{L^2_{\theta_{\eb,s}}}+C_3\eb^{-1}K
\end{equation}
with positive $\gamma$ and $C_3$ which are independent of $t$.
\par
Estimate \eqref{5.bestest} looks similar to \eqref{5.dis}, but there is still an essential difference, namely, the parameter $\eb$ still depends on the initial value $u_0$ through condition \eqref{5.goodind} (in order to initialize the estimates, we need to take $\eb$ satisfying \eqref{3.ebfin} or something similar). Thus, in order to finish the proof of the theorem, we need to show that, for large times, the parameter $\eb$ can be increased and finally made independent of $u_0$.
\par
Let $\eb=\eb_1$ be such that
\begin{equation}\label{5.leftright}
\mu(1+\|u_0\|_{L^2_b}+c^{3/2}+\|g\|_{L^2_b})^{-2}\le \eb_1\le \eb_0=\mu K^{-1}
\end{equation}
and assumption \eqref{5.goodind} be satisfied. Then, according to \eqref{5.bestest}, there exists $T=T(\|u_0\|_{L^2_b})$ such that
\begin{equation}\label{5.dis1}
\|u(t)\|_{L^2_{\theta_{\eb_1,s}}}\le 2C_3\eb_1^{-1}K,\ \ t\ge T.
\end{equation}
Assume now that $\eb_1\le \eb_0/2$. Then, the new $\eb_2=2\eb_1$ satisfies
$$
\eb_2^2\|u(T)\|_{L^2_{\theta_{\eb_2,s}}}^2\le 4\eb_1^2\|u(T)\|^2_{L^2_{\theta_{\eb_1,s}}}\le 8C_3\eb_1 K=8C_3\mu\eb_1\eb_0^{-1}\le 4C_3\mu
$$
and, for sufficiently small $\mu$, condition \eqref{5.goodind} is satisfied. Thus, we may apply estimate \eqref{5.bestest} starting from $t=T$ and using $\eb=\eb_2=2\eb_1$. Repeating this procedure finitely many times if necessary, we finally prove that, for every $u_0$, there exists
$T=T(\|u_0\|_{L^2_b})$ such that
\begin{equation}\label{5.almost}
\|u(t)\|^2_{L^2_{\theta_{\eb,s}}}\le 2C_3\eb^{-1}K,\ \ \eb_0/2\le \eb\le\eb_0,\ \ t\ge T.
\end{equation}
The fact that $T$ is also independent of $s\in\R$ is guaranteed by the left inequality of \eqref{5.leftright}. Taking the supremum over $s\in\R$ from both sides of \eqref{5.almost} and using \eqref{1.wb} and the definition of $\eb_0$, we see that
$$
\|u(t)\|^2_{L^2_b}\le C(1+c^{3/2}+\|g\|_{L^2_b})^4,\ \ t\ge T.
$$
This estimate, which establishes the existence of an absorbing ball for the solution semigroup $S(t)$ in $[L^2_b(\Omega)]^2$ is equivalent to the desired dissipative estimate \eqref{5.dis}. Thus, Theorem \ref{Th5.dis} is proved.
\end{proof}
At the next step, we establish the uniformly local analogue of the standard $L^2\to H^1$ smoothing property for the solutions of the Navier-Stokes problem.
\begin{theorem}\label{Th5.smo} Let the assumptions of Theorem \ref{Th3.main} hold and let $u(t)$ be the energy solution of the Navier-Stokes problem \eqref{3.eqmain}. Then, the following estimate holds:
\begin{equation}\label{5.smo1}
t^{1/2}\|u(t)\|_{W^{1,2}_b}+\|t^{1/2}\Dx u\|_{L^2_b((0,1)\times\Omega)}\le Q(\|u_0\|_{L^2_b}),\ \ t\in[0,1]
\end{equation}
for some monotone function depending on $c$, but independent of $t$ and $u_0$. Moreover, if in addition $u_0\in [W^{1,2}_b(\Omega)]^2$, then the following dissipative estimate holds:
\begin{equation}\label{5.regh1}
\|u(t)\|_{W^{1,2}_b}\le Q(\|u_0\|_{W^{1,2}_b})e^{-\alpha t}+Q(\|g\|_{L^2_b}),
\end{equation}
where the positive constant $\alpha$ and the monotone function $Q$ are independent of $t$ and $u_0$.
\end{theorem}
\begin{proof} The proof of this theorem is identical to the one given in \cite{ZelikGlasgow}. Since, in contrast to many results given above, the auxiliary equation \eqref{2.aug} is not used there, this proof is not affected by the mistake related with the auxiliary equation mentioned in the introduction and, therefore, remains correct. Nevertheless, to the convenience of the reader, we give a sketch of this proof below. To this end, we need the following lemma.
\begin{lemma}\label{Lem5.com} Let $\psi_s$, $s\in\R$ be the cut-off functions introduced in the proof of Theorem \ref{Th4.unique} (with some fixed $\mu>0$ which is not important here) and let also $u\in\Cal H_b\cap[W^{2,2}_b(\Omega)]^2$. Then, for sufficiently small $\eb>0$, the following estimate holds:
\begin{equation}\label{5.est51}
\|\psi_s u\|_{W^{2,2}}^2\le C\|\psi_s\Pi\Dx u\|^2_{L^2}+\delta\|u\|^2_{W^{2,2}_{\theta_{\eb,s}}}+C_\delta\|u\|^2_{L^2_{\theta_{\eb,s}}},
\end{equation}
where $\delta>0$ can be arbitrarily small and the constants $C$ and $C_\delta$ are independent of $s$ and $u$. Moreover, for every $u\in [L^2_b(\Omega)]^2$, the following commutation relation holds:
\begin{equation}\label{5.est52}
\|(\psi_s\circ\Pi-\Pi\circ\psi_s)u\|_{W^{1,2}_{\theta_{\eb,s}^{-2}}}\le C\|u\|_{L^2_{\theta_{\eb,s}}},
\end{equation}
where the constant $C$ is also independent of $s$.
\end{lemma}
The proof of this lemma is given \cite{ZelikGlasgow}.
\par
We will derive below estimates \eqref{5.smo1} and \eqref{5.regh1} only for the case $c=0$ (the general case of non-zero flux can be reduced to it exactly as in Theorem \ref{Th3.main}). To this end, following \cite{ZelikGlasgow}, we multiply equation \eqref{4.time} by
$$
\Nx(\psi_s^2\Nx u)=\psi_s^2\Dx u+2\psi_s\psi_s'\partial_{x_1}u
$$
and integrate over $x\in\Omega$. This gives
\begin{equation}\label{5.h1huge}
\frac12\frac d{dt}\|\Nx(\psi_s u)\|^2_{L^2}+(\Pi\Dx u,\Nx(\psi_s^2\Nx u))=(\Pi[(u,\Nx)u],\Nx(\psi_s^2\Nx u))-(\Pi g,\Nx(\psi_s^2\Nx u)).
\end{equation}
We give below the estimates of the most complicated terms which contain $\psi_s^2\Dx u$ in the right-hand side of \eqref{5.h1huge}. The remaining terms which contain $2\psi_s\psi'_s\partial_{x_1}u$ are lower order and can be estimated in a similar but simpler way (we leave the details to the reader, see also \cite{ZelikGlasgow}). First, denoting $Lu:=(\psi^2_s\circ\Pi-\Pi\circ\psi_s^2)\Dx u$ and integrating by parts, we get
$$
(\Pi\Dx u,\psi_s^2\Dx u)=\|\psi_s\Pi\Dx u\|^2_{L^2}-(\Dx u,\Pi Lu)=\|\psi_s\Pi\Dx u\|^2_{L^2}+(\Nx u,\Nx( \Pi Lu))-(\partial_n u,\Pi L u)_{\partial\Omega},
$$
where $(u,v)_{\partial\Omega}=\int_{\partial_\Omega}u.v\, dS$. Using  the regularity of the Leray projector (see \eqref{1.wp}), the weighted trace theorem $H^s_\theta(\Omega)\subset L^2_\theta(\partial\Omega)$ for $s>\frac12$ (see e.g., \cite{ZelikCPAM} or \cite{ZelikGlasgow}), estimate \eqref{5.est52} and the weighted interpolation inequality, we get
\begin{multline*}
|(\Nx u,\Nx( \Pi Lu))|+|(\partial_n u,\Pi L u)_{\partial\Omega}|\le\\\le C\|\Nx u\|_{L^2_{\theta_{\eb,s}}}\|u\|_{W^{2,2}_{\theta_{\eb,s}}}+C\|u\|_{W^{7/4,2}_{\theta_{\eb,s}}}\|u\|_{W^{2,2}_{\theta_{\eb,s}}}\le \delta\|u\|^2_{W^{2,2}_{\theta_{\eb,s}}}+C_\delta\|u\|^2_{L^2_{\theta_{\eb,s}}}.
\end{multline*}
This estimate, together with \eqref{5.est51} gives
\begin{equation}\label{5.laplace}
(\Pi\Dx u,\psi_s^2\Dx u)\ge \kappa\|\psi_s u\|^2_{W^{2,2}}-\delta\|u\|^2_{W^{2,2}_{\theta_{\eb,s}}}-C_\delta\|u\|^2_{L^2_{\theta_{\eb,s}}},
\end{equation}
where $\delta>0$ can be arbitrarily small and the positive constants $\kappa$ and $C_\delta$ are independent of $s$.
\par
Second, we transform the nonlinear term as follows:
\begin{equation}\label{5.nonlinear}
(\Pi[(u,\Nx)u],\psi_s^2\Dx u)=((u,\Nx)(\psi_s u),\psi_s\Pi\Dx u)-((u,\Nx)u,Lu)-(u_1\psi'_su,\psi_s\Pi\Dx).
\end{equation}
Using estimate \eqref{5.est51}, the first term in the right-hand side can be estimated as follows:
$$
|((u,\Nx)(\psi_s u),\psi_s\Dx u)|\le C\|(u,\Nx)(\psi_su)\|_{L^2}^2+\frac\kappa4\|\psi_su\|^2_{W^{2,2}}+\delta\|u\|^2_{W^{2,2}_{\theta_{\eb,s}}}+C_\delta\|u\|_{L^2_{\theta_{\eb,s}}}^2.
$$
and the first term in the right-hand side of this inequality is controlled by the Ladyzhenskaya inequality:
\begin{multline*}
C\|(u,\Nx)(\psi_su)\|_{L^2}^2\le C\|u\|_{L^4(\Omega_{s-1,s+2})}^2\|\Nx(\psi_s u)\|^2_{L^4}\le\\\le C\|u\|^2_{L^2_b}\|\Nx u\|^2_{L^2(\Omega_{s-1,s+2})}\|\Nx(\psi_s u)\|^2_{L^2}+\frac\kappa4\|\psi_s u\|^2_{W^{2,2}}.
\end{multline*}
The last two terms in the right-hand side of \eqref{5.nonlinear} are lower order and easier to estimate:
$$
|(u,\Nx)u,Lu)|+|(u_1\psi'_su,\psi_s\Pi\Dx)|\le C\|u\|^2_{L^4_{\theta_{\eb,s}}}\|u\|_{W^{2,2}_{\theta_{\eb,s}}}\le C\|u\|^{3/2}_{L^2_{\theta_{\eb,s}}}\|u\|^{3/2}_{W^{2,2}_{\theta_{\eb,s}}}\le C_\delta\|u\|^6_{L^2_{\theta_{\eb,s}}}+\delta\|u\|^2_{W^{2,2}_{\theta_{\eb,s}}}
$$
and, finally, we end up with the following estimate of the nonlinear term:
\begin{multline}\label{5.nonlinear1}
|(\Pi[(u,\Nx)u],\psi_s^2\Dx u)|\le C\|u\|^2_{L^2_b}\|\Nx u\|^2_{L^2(\Omega_{s-1,s+2})}\|\Nx(\psi_su)\|^2_{L^2}+\\+\frac\kappa4\|\psi_s u\|^2_{W^{2,2}}+\delta\|u\|^2_{W^{2,2}_{\theta_{\eb,s}}}+C_\delta\|u\|^2_{L^2_{\theta_{\eb,s}}}(1+\|u\|^4_{L^2_{\theta_{\eb,s}}}).
\end{multline}
Inserting the estimates \eqref{5.laplace} and \eqref{5.nonlinear1} (together with the analogous estimates for the lower order terms) to \eqref{5.h1huge}, we arrive at
\begin{multline}\label{5.h1huge1}
\frac d{dt}\|\Nx (\psi_su)\|^2_{L^2}+\kappa\|\psi_su\|^2_{W^{2,2}}\le C\|u\|^2_{L^2_b}\|\Nx u\|^2_{L^2(\Omega_{s-1,s+2})}\|\Nx(\psi_su)\|^2_{L^2}+\\+\delta\|u\|^2_{W^{2,2}_{\theta_{\eb,s}}}+
C_\delta\|u\|^2_{L^2_{\theta_{\eb,s}}}(1+\|u\|^4_{L^2_{\theta_{\eb,s}}}).
\end{multline}
Multiplying this estimate on $t$ and using the interpolation inequality $\|v\|_{H^1}^2\le C\|v\|_{L^2}\|v\|_{H^2}$, we have
\begin{multline}\label{5.h1huge2}
\frac d{dt}(t\|\Nx (\psi_su)\|^2_{L^2})+\kappa t\|\psi_su\|^2_{W^{2,2}}\le C\|u\|^2_{L^2_b}\|\Nx u\|^2_{L^2(\Omega_{s-1,s+2})}(t\|\Nx(\psi_su)\|^2_{L^2})+\\+\delta t\|u\|^2_{W^{2,2}_{\theta_{\eb,s}}}+
C_\delta\|u\|^2_{L^2_{\theta_{\eb,s}}}(1+\|u\|^4_{L^2_{\theta_{\eb,s}}}).
\end{multline}
Applying the Gronwall inequality (with respect to the function $t\to t\|\Nx(\psi_su)\|^2_{L^2}$) to this estimate and using that $\|u(t)\|_{L^2_b}$ and $\|\Nx u\|_{L^2_b((0,1)\times\Omega)}$ are controlled by the energy norm of the initial data (due to estimate \eqref{3.00est}), we have
\begin{equation}
t\|\Nx(\psi_s u(t))\|^2_{L^2}+\int_0^tt\|\psi_su(t)\|^2_{W^{2,2}}\,dt\le \delta\int_0^tt\|u(t)\|^2_{W^{2,2}_{\theta_{\eb,x_0}}}\,dt+Q_\delta(\|u_0\|_{L^2_b}), \ t\in[0,1],
\end{equation}
where $\delta>0$ is arbitrary and the monotone function $Q_\delta$ is independent of $t$ and $u$. Multiplying this inequality by $\theta_{\eb,\tau}(s)$, $\tau\in\R$, integrating over $s\in\R$ and using \eqref{1.wequiv} and \eqref{4.obvious}, we see that, for sufficiently small $\delta>0$,
\begin{equation}\label{5.h1huge3}
t\|\Nx u(t))\|^2_{L^2_{\theta_{\eb,\tau}}}+\int_0^tt\|u(t)\|^2_{W^{2,2}_{\theta_{\eb,\tau}}}\,dt\le Q(\|u_0\|_{L^2_b}), \ t\in[0,1]
\end{equation}
and, taking the supremum over $\tau\in\R$ from both sides of this inequality, we prove the desired estimate \eqref{5.smo1}.
\par
To verify \eqref{5.regh1}, we  observe that if we know that $u_0\in[W^{1,2}_b(\Omega)]^2$, we need not to multiply \eqref{5.h1huge1} by $t$ and may apply the Gronwall inequality directly to \eqref{5.h1huge1}. Then, arguing as before, we end up with
$$
\|\Nx u(t))\|^2_{L^2_{\theta_{\eb,\tau}}}+\int_0^t\|u(t)\|^2_{W^{2,2}_{\theta_{\eb,\tau}}}\,dt\le Q(\|u_0\|_{L^2_b}), \ t\in[0,1]
$$
which proves \eqref{5.regh1} for $t\in[0,1]$. To verify it for $t\ge1$, it is sufficient to combine the smoothing estimate \eqref{5.smo1} of the form
$$
\|u(t)\|_{W^{1,2}_b}\le Q(\|u(t-1)\|_{L^2_b}),\ t\ge1
$$
with the dissipative estimate \eqref{3.00est} for $u(t-1)$. Thus, estimate \eqref{5.regh1} is proved and the theorem is also proved.
\end{proof}
We conclude the section by establishing the analogue of Theorem \ref{Th5.smo} for more regular space $W^{2,2}_b(\Omega)$.

\begin{corollary}\label{Cor5.h2smo} Let the assumptions of Theorem \ref{Th3.main} hold. Then, the energy solution $u(t)$ belongs to the space $[W^{2,2}_b(\Omega)]^2$ and the following estimate holds:
\begin{equation}\label{5.h2smo}
t\|u(t)\|_{W^{2,2}_b}\le Q(\|u_0\|_{L^2_b}),\ \ t\in[0,1],
\end{equation}
where the monotone function $Q$ is independent of $u$ and $t$. Moreover, if in addition $u_0\in [W^{2,2}_b(\Omega)]^2$, then the following dissipative estimate holds:
\begin{equation}\label{5.dish2}
\|u(t)\|_{W^{2,2}_b}\le Q(\|u_0\|_{W^{2,2}_b})e^{-\alpha t}+Q(\|g\|_{L^2_b}),
\end{equation}
where the positive constant $\alpha$ and the monotone function $Q$ are independent of $t\ge0$ and $u$.
\end{corollary}
\begin{proof} Indeed, arguing exactly as at the end  the proof of Theorem \ref{Th5.smo}, we see that the dissipative estimate \eqref{5.dish2} is a formal corollary of the smoothing estimate \eqref{5.h2smo} and estimates \eqref{3.00est} and \eqref{4.badest}. Thus, we only need to prove the smoothing property \eqref{5.h2smo}. To this end, we differentiate equation \eqref{3.eqmain} by $t$ and  denote $v(t)=\Dt u(t)$. Then, estimate \eqref{4.derest} (applied from the initial time moment $t=\tau$ instead of $t=0$ and used on the time interval $t\in(\tau,T)$, $T\le1$) gives
\begin{equation}
\|v(T)\|_{L^2_{\theta_{\eb,s}}}^2+\int_\tau^T\|\Nx v(t)\|^2_{L^2_{\theta_{\eb,s}}}\,dt\le C\|\Dt u(\tau)\|^2_{L^2_{\theta_{\eb,s}}},
\end{equation}
where $C$ depends on the $L^2_b$-norm of the initial data $u_0$, but is independent of $\tau$, $T$ and $s$. Multiplying this estimate by $\tau$ and integrating over $\tau\in[0,T]$, we arrive at
\begin{equation}\label{5.dtest}
T^2\|v(T)\|_{L^2_{\theta_{\eb,s}}}^2+\int_0^Tt^2\|\Nx v(t)\|^2_{L^2_{\theta_{\eb,s}}}\,dt\le 2C\int_0^Tt\|\Dt u(t)\|^2_{L^2_{\theta_{\eb,s}}}\,dt.
\end{equation}
To estimate the integral in the right-hand side, we use \eqref{4.time}. This gives
$$
\int_0^1t\|\Dt u(t)\|^2_{L^2_{\theta_{\eb,s}}}\le C\int_0^1t\|u(t)\|^2_{W^{2,2}_{\theta_{\eb,s}}}+C\int_0^t t\|(u(t),\Nx)u(t)\|^2_{L^2_{\theta_{\eb,s}}}\,dt.
$$
Using the Ladyzhenskaya inequality together with \eqref{1.wequiv}, the smoothing property \eqref{5.smo1} and the the energy estimate \eqref{3.00est}, we get
\begin{multline*}
\int_0^t t\|(u(t),\Nx)u(t)\|^2_{L^2_{\theta_{\eb,s}}}\,dt\le C\int_0^1\int_{s\in\R}\theta_s(y)t\|u(t)\|^2_{L^4(\Omega_y)}\|\Nx u(t)\|^2_{L^4(\Omega_y)}\,dy\,dt\le\\\le C_1\|u\|_{C(0,T;L^2_b)}\|t^{1/2}\Nx u\|_{C(0,T;L^2_b)}\int_0^1t^{1/2}\theta_{\eb,s}(y)\|u(t)\|_{W^{1,2}(\Omega_y)}\|u(t\|_{W^{2,2}(\Omega_y)}\,dy\,dt\le\\\le C_2\int_0^1(\|u(t)\|^2_{L^2_{\theta_{\eb,s}}}+t\|u(t)\|^2_{L^2_{\theta_{\eb,s}}})\,dt\le Q(\|u_0\|_{L^2_b}).
\end{multline*}
Thus, due to \eqref{5.dtest}, we have the smoothing property for the time derivative $\Dt u(t)$, namely,
\begin{equation}\label{5.dtest1}
t\|\Dt u(t)\|_{L^2_b}\le Q(\|u_0\|_{L^2b}),\ \ t\in[0,1]
\end{equation}
and we only need to derive from this estimate the desired estimate for the $W^{2,2}_b$-norm of $u(t)$. To this end, we consider equation \eqref{4.time} as a stationary Navier-Stokes problem (with the right-hand side $g-\Dt u(t)$) for every fixed $t$ and use the maximal regularity of the Stokes operator in the uniformly local spaces as well as the Ladyzhenskaya inequality. This gives
\begin{multline*}
t\|u(t)\|_{W^{2,2}_b}\le C\|g\|_{L^2_b}+t\|\Dt u(t)\|_{L^2_b}+C t\|(u,\Nx) u\|_{L^2_b}\le Q(\|u_0\|_{L^2_b})+Q(\|g\|_{L^2_b})+\\+C t\|u\|^{1/2}_{W^{2,2}_b}\|u\|_{L^2_b}^{1/2}\|\Nx u(t)\|_{L^2_b}\le \frac12t\|u(t)\|_{W^{2,2}_b}+\\+Ct\|\Nx u(t)\|_{L^2_b}^{2}+Q(\|u_0\|_{L^2_b})+Q(\|g\|_{L^2_b})\le \frac12t\|u(t)\|_{W^{2,2}_b}+Q(\|u_0\|_{L^2_b})+Q(\|g\|_{L^2_b}).
\end{multline*}
This estimates give the desired control on $t\|u(t)\|_{W^{2,2}_b}$ and finishes the proof of the corollary.
\end{proof}
\begin{remark} The standard bootstrapping arguments show that the  actual regularity of a weak solution $u(t)$ is restricted by the regularity of the external forces $g$ and the initial data only and, for instance if $g\in C^\infty(\Omega)$, the solution $u(t)\in C^\infty(\Omega)$ for all $t>0$.
\par
Moreover, the obtained in this section results (dissipative and smoothing estimates) allow us to verify the existence of the so-called locally compact global attractor $\Cal A$ for that equation and derive the estimate for its size in $L^2_b(\Omega)$:
$$
\|\Cal A\|_{L^2_b}\le C(1+ c^3+\|g\|_{L^2_b}^2)
$$
which is exactly the same as in \cite{ZelikGlasgow}. Thus, we see that despite the inaccuracy with estimating the solutions of the auxiliary problem, main results of \cite{ZelikGlasgow} remain correct.
\end{remark}


\begin{thebibliography}{99}
\bibitem{Abels} H. Abels, Boundedness of imaginary powers of the
stokes operator in an infinite layer, J. Evol. Equ., vol. 2, (2002), 439--457.
\bibitem {Ab1}F. Abergel, Attractor for a Navier-Stokes flow in an unbounded domain.
Attractors, inertial manifolds and their approximation
(Marseille-Luminy, 1987). RAIROModґel.Math. Anal. Numґer. {\bf 23}
(1989) 359--370.
\bibitem {Ab2}
 F. Abergel,  Existence and finite dimensionality of the
global attractor for evolution equations on unbounded domains.  J.
Differential Equations  {\bf 83}  (1990) 85--108.  DOI:10.1016/0022-0396(90)90070-6.
\bibitem{MielkeA} A. Afendikov and A. Mielke, Dynamical properties of spatially non-decaying 2D Navier-Stokes flows with Kolmogorov forcing in an infinite strip, J. Math. Fluid Mech., 7 (2005),
51--67.
\bibitem{Amann}
H. Amann,
On the strong solvability of the Navier-Stokes equations,
Jour. Math.Fluid Mechanics,
vol 2, (2000), 16--98.
\bibitem{babin}
A. Babin,
Asymptotic Expansions at infinity of a strongly perturbed
Poiseuille flow, Advances in Soviet Math.,
vol 10, (1992), 1--83.

\bibitem{babin1}
A. Babin,
 The attractor of a Navier-Stokes system in an unbounded channel-like
domain, J. Dynam. Differential Equations,
vol 4, no. 4, (1992), 555--584.
\bibitem{BabinVishik}
A. Babin and M.Vishik,
Attractors of Partial Differential Evolution Equations in an Unbounded
Domain. Proc. Roy. Soc.  Edinburgh Sect. A,
vol. 116, no. 3-4, (1990), 221--243.
\bibitem{BV} A. V. Babin and M. I. Vishik, Attractors of Evolution Equations (Nauka, Moscow, 1989; North Holland,
Amsterdam, 1992).
\bibitem{EfZelCPAM} M. Efendiev and S. Zelik,  The attractor for a nonlinear reaction-diffusion system in an unbounded domain, Comm. Pure Appl. Math., 54 (2001), no. 6, 625--688.
\bibitem{Giga1}  Y. Giga, S. Matsui and O. Sawada, Global existence of two-dimensional Navier-Stokes
flow with nondecaying initial velocity, J. Math. Fluid Mech., 3 (2001), 302--315.
\bibitem{Henry} D. Henry, {\it Geometric theory of semilinear parabolic equations.} Lecture Notes in Mathematics,
840. Springer-Verlag, Berlin–New York, 1981.
\bibitem{L02} P. Lemarie-Rieusset,
Recent developments in the Navier-Stokes problem.
Chapman $\&$ Hall/CRC Research Notes in Mathematics, 431. Chapman \& Hall/CRC, Boca Raton, FL, 2002.
\bibitem{MielkeS}
A. Mielke and G.Schneider,
Attractors for Modulation Equations on Unbounded Domains --
Existence and Comparison,
Nonlinearity,
vol 8, (1995), 743--768.
\bibitem{MirZel} A. Miranville and S. Zelik, Attractors for dissipative partial differential equations in bounded and unbounded
domains. Handbook of differential equations: evolutionary equations. Vol. IV, 103--200, Handb. Differ. Equ.,
Elsevier/North-Holland, Amsterdam, 2008.
\bibitem{ZelPenCPAA} J. Pennant and S. Zelik, Global well-posedness in uniformly local spaces for the Cahn-Hilliard equation in $\R^3$, Comm. Pure Appl. Anal., vol 12, no. 1, (2013), 461--480.
\bibitem{Judovich} S. Revina and V. Yudovich, $L^p$-estimates for the resolvent of the Stokes operator in an infinite cylinder. (Russian) Mat. Sb. 187 (1996), no. 6, 97--118; translation in Sb. Math. 187 (1996), no. 6, 881–902.
\bibitem {Giga2} O. Sawada and Y. Taniuchi, A remark on $L^\infty$-solutions to the 2D Navier-Stokes equations, J. Math. Fluid Mech., 9(2007), 533--542.
\bibitem{temam} R. Temam, Navier-Stokes Equations, Theory and Numerical Analysis (North-Holland, Amsterdam New
York-Oxford, 1977).
\bibitem{temam1} R. Temam, Infnite-Dimensional Dynamical Systems in Mechanics and Physics, Applied Mathematics Series
(Springer, New York-Berlin, 1988; 2nd ed., New York, 1997).
\bibitem{triebel} H. Triebel, Interpolation Theory, Function Spaces, Differential Operators, North-Holland, 1978.

\bibitem{ZelikGlasgow} S. Zelik,  Spatially nondecaying solutions of the 2D Navier-Stokes equation in a strip. Glasg. Math. J., 49 (2007), no. 3, 525--588.
\bibitem{ZelikInst} S. Zelik, Weak spatially nondecaying solutions of 3D Navier-Stokes equations in cylindrical domains. Instability in models connected with fluid flows. II, 255--327, Int. Math. Ser. (N. Y.), 7, Springer, New York, 2008.
\bibitem{ZelikCPAM} S. Zelik,  Attractors of reaction-diffusion systems in unbounded domains and their spatial complexity, Comm. Pure Appl. Math., 56 (2003), no. 5, 584--637.
\bibitem{ZelJMFM} S. Zelik,  Infinite energy solutions for damped Navier-Stokes equations in $\R^2$, Jour. Math. Fluid Mech., (2013).

\end{thebibliography}
\end{document}